\documentclass{amsart}

\usepackage[T1]{fontenc}
\usepackage[utf8]{inputenc}
\usepackage[american]{babel}

\usepackage{amsmath}
\usepackage{amsthm}
\usepackage{amssymb}

\numberwithin{equation}{section}

\theoremstyle{plain}
\newtheorem{thm}{Theorem}[section]
\newtheorem{lem}[thm]{Lemma}
\theoremstyle{remark}
\newtheorem{rem}[thm]{Remark}

\usepackage[scr=boondox]{mathalfa}

\let\op\mathcal
\let\set\mathscr

\newcommand*{\hairspace}{\kern 0.08333em}

\newcommand*{\dual}{+}
\newcommand*{\dualref}{\textsuperscript{+}}

\renewcommand{\Re}{\operatorname{Re}}
\renewcommand{\Im}{\operatorname{Im}}

\newcommand*{\cell}{Q}
\newcommand*{\cellradius}{r_\cell}

\newcommand*{\dd}{\mathop{}\!d}
\newcommand*{\findiff}{\varDelta}

\newcommand*{\meet}{{\wedge}}

\newcommand*{\wc}{{\mkern 2mu\cdot\mkern 2mu}}

\DeclareMathOperator{\diam}{diam}
\DeclareMathOperator{\dist}{dist}
\DeclareMathOperator{\Div}{div}
\DeclareMathOperator{\spec}{spec}

\providecommand{\rbr}[1]{(#1)}
\providecommand{\bigrbr}[1]{\bigl(#1\bigr)}
\providecommand{\Bigrbr}[1]{\Bigl(#1\Bigr)}
\providecommand{\sbr}[1]{[#1]}
\providecommand{\cbr}[1]{\{#1\}}
\providecommand{\bigcbr}[1]{\bigl\{#1\bigr\}}
\providecommand{\abr}[1]{\langle#1\rangle}
\providecommand{\abs}[1]{\lvert#1\rvert}
\providecommand{\bigabs}[1]{\bigl\lvert#1\bigr\rvert}
\providecommand{\norm}[1]{\lVert#1\rVert}
\let\seminorm\sbr

\newcommand*{\B}{\mathbf{B}}
\newcommand*{\C}{\mathbb{C}}
\newcommand*{\N}{\mathbb{N}}
\newcommand*{\R}{\mathbb{R}}
\newcommand*{\Z}{\mathbb{Z}}

\usepackage{mathtools}
\usepackage{graphicx}

\begin{document}

\title[Homogenization for locally periodic elliptic operators]
      {Homogenization for non-self-adjoint\\ locally periodic elliptic operators}
\author{Nikita N. Senik}
\address{Saint~Petersburg State University, Universitetskaya nab.~7/9, Saint~Petersburg 199034, Russia}
\email{\protect\raisebox{-1.6pt}{\protect\includegraphics{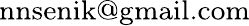}}}
\thanks{The~author was partially funded by Young Russian Mathematics award,
Rokhlin grant and RFBR grant 16-01-00087.}



\begin{abstract}
We study the homogenization problem for matrix strongly
elliptic operators on $L_2(\mathbb R^d)^n$ of the
form~$\mathcal A^\varepsilon=-\operatorname{div}A(x,x/\varepsilon)\nabla$.
The~function~$A$ is Lipschitz in the first variable and periodic
in the second. We do not require that $A^*=A$, so $\mathcal A^\varepsilon$
need not  be self-adjoint. In~this paper,  we  provide, for
small~$\varepsilon$, two terms in the uniform approximation for
$(\mathcal A^\varepsilon-\mu)^{-1}$  and a first term in the
uniform approximation for $\nabla(\mathcal A^\varepsilon-\mu)^{-1}$.
Primary attention is paid to proving sharp-order bounds on the errors
of the approximations.
\end{abstract}

\maketitle

\section{Introduction}

 Homogenization  dates back to the late 1960s, and for more
than fifty years it  has  become a well-established theory. In~the~simplest
case, homogenization deals with asymptotic properties of  solutions
to differential equations with oscillating coefficients.  Given a
periodic (with period~$1$ in each variable) uniformly bounded and
uniformly positive definite function~$A\colon\R^{d}\to\C^{d\times d}$,
consider the differential equation
\begin{equation}
-\Div A\rbr{\varepsilon^{-1}x}\hairspace\nabla u_{\varepsilon}-\mu u_{\varepsilon}=f,\label{eq: Intro | Heterogeneous equation}
\end{equation}
where $\varepsilon>0$, $\mu\in\C\setminus\R_{+}$ and~$f\in L_{2}\rbr{\R^{d}}$.
The~coefficients of the equation are $\varepsilon$\nobreakdash-periodic
and hence rapidly oscillate if $\varepsilon$ is small. In~homogenization
theory one is interested in studying the asymptotic behavior of $u_{\varepsilon}$
as~$\varepsilon$ becomes smaller. It~is a basic fact that, after
passing to a subsequence if necessary, $u_{\varepsilon}$ converges
 to the solution~$u_{0}$ of the differential equation
\begin{equation}
-\Div A^{0}\nabla u_{0}-\mu u_{0}=f\label{eq: Intro | Homogeneous equation}
\end{equation}
with constant~$A^{0}$. Since, in applications, the elliptic operator
on the left side of (\ref{eq: Intro | Heterogeneous equation}) usually
describes a physical process in a highly heterogeneous medium, this
means that, in certain aspects, the process evolves very similar to
that in a homogeneous medium.

It is a basic fact about homogenization theory that $u_{\varepsilon}$
converges to $u_{0}$  in~$L_{2}\rbr{\R^{d}}$; we refer the
reader to \cite{BLP:1978}, \cite{BP:1984} or~\cite{ZhKO:1993}
for the details. Stated differently, the resolvent of $-\Div A\rbr{\varepsilon^{-1}x}\hairspace\nabla$
converges in the strong operator topology to the  resolvent of~$-\Div A^{0}\nabla$.
In~\cite{BSu:2001} (see~also~\cite{BSu:2003}), Birman and~Suslina
proved that, in fact, the resolvent converges in norm. Moreover,
they  found a sharp-order bound on the rate of convergence. Since
that time there have been a number of interesting further results
in this direction~\textendash{} see \cite{Gri:2004}, \cite{Gri:2006},
\cite{Zh:2005}, \cite{ZhPas:2005}, \cite{Bor:2008}, \cite{KLS:2012},
\cite{Su:2013-1}, \cite{Su:2013-2}, \cite{ChC:2016} and~\cite{ZhPas:2016},
to name  a few.

Here we focus on a more general problem than the periodic one in~(\ref{eq: Intro | Heterogeneous equation}).
Let $A=\cbr{A_{kl}}$ with $A_{kl}\colon\R^{d}\times\R^{d}\to\C^{n\times n}$
being uniformly bounded functions that are Lipschitz in the first
variable and periodic in the second (see~Section~\ref{sec: Original operator}
for a precise definition). Consider the operator~$\op A^{\varepsilon}$
on the complex space~$L_{2}\rbr{\R^{d}}^{n}$ given~by
\[
\op A^{\varepsilon}=-\Div A\rbr{x,\varepsilon^{-1}x}\hairspace\nabla=-\sum_{k,l=1}^{d}\partial_{k}A_{kl}\rbr{x,\varepsilon^{-1}x}\hairspace\partial_{l}.
\]
The~coefficients  now depend not only on the ``fast'' variable,~$\varepsilon^{-1}x$,
but also on the ``slow'' one,~$x$. Assume that, for all $\varepsilon$
in some neighborhood of~$0$, the operator~$\op A^{\varepsilon}$
is coercive and furthermore the constants in the coercivity bound
are independent of~$\varepsilon$. Then $\op A^{\varepsilon}$ is
strongly elliptic for such $\varepsilon$ and there is a sector containing
the spectrum of~$\op A^{\varepsilon}$. In~this paper, we will
obtain approximations for $\rbr{\mathcal{A}^{\varepsilon}-\mu}^{-1}$
and~$\nabla\rbr{\mathcal{A}^{\varepsilon}-\mu}^{-1}$ (with~$\mu$
outside the sector) in the operator norm  and prove that
\begin{gather}
\norm{\rbr{\mathcal{A}^{\varepsilon}-\mu}^{-1}-\rbr{\mathcal{A}^{0}-\mu}^{-1}}_{L_{2}\to L_{2}}\le C\varepsilon,\label{est: Intro | Convergence}\\
\norm{\rbr{\mathcal{A}^{\varepsilon}-\mu}^{-1}-\rbr{\mathcal{A}^{0}-\mu}^{-1}-\varepsilon\op C_{\mu}^{\varepsilon}}_{L_{2}\to L_{2}}\le C\varepsilon^{2}\label{est: Intro | Approximation in L2}
\end{gather}
and
\begin{equation}
\norm{\nabla\rbr{\mathcal{A}^{\varepsilon}-\mu}^{-1}-\nabla\rbr{\mathcal{A}^{0}-\mu}^{-1}-\varepsilon\nabla\op K_{\mu}^{\varepsilon}}_{L_{2}\to L_{2}}\le C\varepsilon,\label{est: Intro | Approximation in H1}
\end{equation}
the~estimates being sharp with respect to the order (see Theorems~\ref{thm: Convergence and Approximation with 1st corrector}
and~\ref{thm: Approximation with 2d corrector}). The~effective
operator~$\mathcal{A}^{0}$ is  of the same form as~$\mathcal{A}^{\varepsilon}$,
but its coefficients  depend only on the slow variable. In~contrast,
the correctors~$\op K_{\mu}^{\varepsilon}$ and~$\op C_{\mu}^{\varepsilon}$
involve rapidly oscillating functions as well. The~first of these
plays the role of the traditional corrector and differs from the
latter in that it involves a smoothing operator. The~idea of using
a smoothing to regularize the traditional corrector is due to Griso,
see~\cite{Gri:2002}. The~other corrector has no analogue in classical
theory and was first presented in~\cite{BSu:2005} for purely periodic
operators. Assume for simplicity that $A^{*}=A$. Then $\op C_{\mu}^{\varepsilon}$
has the form
\[
\op C_{\mu}^{\varepsilon}=\rbr{\op K_{\mu}^{\varepsilon}-\op L_{\mu}}-\op M_{\mu}^{\varepsilon}+\rbr{\op K_{\mu}^{\varepsilon}-\op L_{\mu}}^{*}
\]
(see~Section~\ref{sec: Correctors}). What is interesting here
is that an analog of $\op C_{\mu}^{\varepsilon}$ for periodic operators,
while looking similar to this one, does not include the term~$\op M_{\mu}^{\varepsilon}$,
see~\cite{Se:2017-2}. In~fact,  one cannot remove $\op M_{\mu}^{\varepsilon}$
from $\op C_{\mu}^{\varepsilon}$ if the estimate~(\ref{est: Intro | Approximation in L2})
is to remain true, see Remark~\ref{rem: Cannot be removed}  for
examples. So this term is  a special feature of non-periodic problems.

The~results of the present paper extend the author's work~\cite{Se:2017-2}
on periodic elliptic problems, where we studied non-self-adjoint scalar
 operators whose coefficients were periodic in some variables and
Lipschitz in the others. Put differently, the fast and slow variables
were separated in the sense that $A^{\varepsilon}\rbr x=A\rbr{x_{1},\varepsilon^{-1}x_{2}}$,
where~$x=\rbr{x_{1},x_{2}}$.  We proved analogs of the estimates~(\ref{est: Intro | Convergence})\textendash (\ref{est: Intro | Approximation in H1}),
yet the correctors were slightly different, see Remark~\ref{rem: Correctors in periodic case}
below. It~should be pointed out that the operators in \cite{Se:2017-2}
were allowed to involve lower-order terms with quite general coefficients.

Previous results on uniform approximations for  locally periodic
elliptic operators  are due to, on the one hand, Borisov and, on
the other hand, Pastukhova and~Tikhomirov. In~\cite{Bor:2008} Borisov
established the estimates~(\ref{est: Intro | Convergence}) and~(\ref{est: Intro | Approximation in H1})
for certain matrix self-adjoint operators with smooth coefficients.
In~the~paper~\cite{PasT:2007} of Pastukhova and~Tikhomirov, similar
results were proved for scalar self-adjoint operators with rough
coefficients (although their techniques also apply to non-self-adjoint
problems). As~far as I know, the estimate~(\ref{est: Intro | Approximation in L2})
in the locally periodic settings was not obtained even for the
simplest cases.

To~prove the estimates, we develop the ideas of~\cite{Se:2017-2}.
In~the~first step we establish a variant of the resolvent identity
that involves the resolvents of the original and the effective operators
and a corrector (see~Section~\ref{sec: Proof of the main results}).
This combination comes as no surprise, for it is well known 
that the effective operator and a corrector form a  first approximation
to the original operator (see, e.g.,~\cite{BLP:1978} or~\cite{ZhKO:1993}).
When this is done, all the desired estimates will follow at once.
However, we cannot use the same technique as in \cite{Se:2017-2},
so the identity is proved by different means. The~point is that
the technique depends heavily on the smoothing operator that has been
chosen. In~the~case of periodic operators, the smoothing was based
on the Gelfand transform; but it is not as convenient now. To~my
knowledge, no natural smoothing for operators with locally periodic
coefficients is known, so we choose the Steklov smoothing operator,
which is the most simple and has proved to be quite useful; see~\cite{Zh:2005}
and~\cite{ZhPas:2005}, where that smoothing first appeared in the
context of homogenization, as well as \cite{PasT:2007}, \cite{Su:2013-1}
and~\cite{Su:2013-2}. We remark that a very similar smoothing had
been used earlier in~\cite{Gri:2002} and~\cite{Gri:2004} (see
also~\cite{Gri:2006}). Our technique is strongly influenced by
all these works.

I~believe that the same method can be of use for locally periodic
problems on domains with Dirichlet or Neumann boundary conditions
as well.

It~is also worth noting that, once the estimates~(\ref{est: Intro | Convergence})\textendash (\ref{est: Intro | Approximation in H1})
are verified, a limiting argument will give similar results for operators
whose coefficients are H\"{o}lder continuous in the first variable,
see Remark~\ref{rem: Holder coefficients}. These results, together
with the results stated here, have been announced in~\cite{Se:2017-3}.

The~plan of the paper is as follows. Section~\ref{sec: Notation}
contains basic definitions and notation. In~Section~\ref{sec: Original operator}
we introduce the original operator. We  study the effective operator
in Section~\ref{sec: Effective operator} and correctors in Section~\ref{sec: Correctors}.
Section~\ref{sec: Main results} states the main results. Section~\ref{sec: Proof of the main results}
is the  core of the paper, where we first prove the identity and
then complete the proofs.

\section{\label{sec: Notation}Notation}

The~symbol~$\norm{\wc}_{U}$ will stand for the norm on a normed
space~$U$. If $U$ and~$V$ are Banach spaces, then $\B\rbr{U,V}$
is the Banach space of bounded linear operators from $U$ to~$V$.
When~$U=V$, the space~$\B\rbr U=\B\rbr{U,U}$ becomes a Banach
algebra with the identity map~$\op I$. The~norm and the inner
product on $\C^{n}$ are denoted by $\abs{\wc}$ and~$\abr{\wc,\wc}$,
respectively. We shall often identify $\B\rbr{\C^{n},\C^{m}}$
and~$\C^{m\times n}$.

Let $\Sigma$ be a domain in $\R^{d}$ and $U$ a Banach space. The~space~$C^{0,1}\rbr{\bar{\Sigma};U}$
consists of those  uniformly continuous functions~$u\colon\Sigma\to U$
for which
\[
\norm u_{C^{0,1}\rbr{\bar{\Sigma};U}}=\norm u_{C\rbr{\bar{\Sigma};U}}+\seminorm u_{C^{0,1}\rbr{\bar{\Sigma};U}}<\infty,
\]
where  $\norm u_{C\rbr{\bar{\Sigma};U}}=\sup_{x\in\Sigma}\norm{u\rbr x}_{U}$
and
\[
\seminorm u_{C^{0,1}\rbr{\bar{\Sigma};U}}=\sup_{\begin{subarray}{c}
x_{1},x_{2}\in\Sigma,\\
x_{1}\ne x_{2}
\end{subarray}}\frac{\norm{u\rbr{x_{2}}-u\rbr{x_{1}}}_{U}}{\abs{x_{2}-x_{1}}}.
\]
We will use the notation~$\norm{\wc}_{C^{0,1}}$, $\norm{\wc}_{C}$
and~$\seminorm{\wc}_{C^{0,1}}$ as shorthand for $\norm{\wc}_{C^{0,1}\rbr{\bar{\Sigma};U}}$,
$\norm{\wc}_{C\rbr{\bar{\Sigma};U}}$ and~$\seminorm{\wc}_{C^{0,1}\rbr{\bar{\Sigma};U}}$
when the context makes clear which $\Sigma$ and~$U$ are meant.

The~symbol~$L_{p}\rbr{\Sigma;U}$ stands for the $L_{p}$\nobreakdash-space
of strongly measurable functions on $\Sigma$ with values in~$U$.
 In~case~$U=\C^{n}$, we write $\norm{\wc}_{p,\Sigma}$ for the
norm on $L_{p}\rbr{\Sigma}^{n}$ and~$\rbr{\wc,\wc}_{\Sigma}$ for
the inner product on~$L_{2}\rbr{\Sigma}^{n}$. We let $W_{p}^{m}\rbr{\Sigma}^{n}$
denote the usual Sobolev space of $\C^{n}$\nobreakdash-valued functions
on $\Sigma$ and~$\rbr{W_{p}^{m}\rbr{\Sigma}^{n}}^{*}$,  its dual
space under the pairing~$\rbr{\wc,\wc}_{\Sigma}$. If~$C_{c}^{\infty}\rbr{\Sigma}^{n}$
is dense in $W_{p}^{m}\rbr{\Sigma}^{n}$, then $W_{\smash[t]{\cramped{p^{\dual}}}}^{-m}\rbr{\Sigma}^{n}=\rbr{W_{p}^{m}\rbr{\Sigma}^{n}}^{*}$,
where $p^{\dual}$ is the exponent conjugate to~$p$.

Let $\cell$ be the closed cube in $\R^{d}$ with center~$0$
and side length~$1$, sides being parallel to the axes. Then 
$\tilde W_{p}^{m}\rbr{\cell}^{n}$ denotes the completion of $\tilde C^{m}\rbr{\cell}^{n}$
in the $W_{p}^{m}$\nobreakdash-norm. Here $\tilde C^{m}\rbr{\cell}$
is the class of $m$\nobreakdash-times continuously differentiable
functions  on $\cell$ whose periodic extension to $\R^{d}$ enjoys
the same smoothness. Notice that $\tilde L_{p}\rbr{\cell}^{n}$ coincides
with the space of all periodic functions in~$L_{p,\text{loc}}\rbr{\R^{d}}^{n}$.
The~spaces~$\tilde W_{p}^{m}\rbr{\R^{d}\times\cell}^{n}$ and~$\tilde C^{m}\rbr{\R^{d}\times\cell}^{n}$
are defined in a similar fashion. If~$p=2$, we write $H^{m}$ for
$W_{p}^{m}$, $H^{-m}$ for $W_{p}^{-m}$,~etc. The~symbol~$\tilde H_{0}^{m}\rbr{\cell}^{n}$
will stand for the subspace of functions in $\tilde H^{m}\rbr{\cell}^{n}$
with mean value zero. Any~$u\in\tilde H_{0}^{1}\rbr{\cell}^{n}$ satisfies
the Poincar\'{e} inequality
\begin{equation}
\norm u_{2,\cell}\le\rbr{2\pi}^{-1}\norm{Du}_{2,\cell},\label{est: Poincar=0000E9's inequality}
\end{equation}
as~can be seen by using Fourier series. Here and below, $D=-i\nabla$.

We will often use the notation~$\alpha\lesssim\beta$ to mean that
that there is a constant~$C$, depending only on some fixed parameters
(these are listed in Theorems~\ref{thm: Convergence and Approximation with 1st corrector}
and~\ref{thm: Approximation with 2d corrector}), such that~$\alpha\le C\beta$.

\section{\label{sec: Original operator}Original operator}

Let each~$A_{kl}$ be a function in~$C^{0,1}\rbr{\bar{\R}^{d};\tilde L_{\infty}\rbr{\cell}}^{n\times n}$.
Then $A=\cbr{A_{kl}}$ may be thought of as a bounded mapping~$A\colon\R^{d}\times\R^{d}\to\B\rbr{\C^{d\times n}}$
that is Lipschitz in the first variable and periodic in the second.
 As~is~well known, for any function~$u\colon\R^{d}\times\R^{d}\to L_{2}\rbr{\cell}$
satisfying the Carath\'{e}odory condition (i.e., the requirement
of continuity with respect to the first variable and measurability
with respect to the second)  the map $\tau^{\varepsilon}u\colon\R^{d}\to L_{2}\rbr{\cell}$
defined for $x\in\R^{d}$ and~$z\in\cell$~by
\begin{equation}
\tau^{\varepsilon}u\rbr{x,z}=u\rbr{x,\varepsilon^{-1}x,z},\label{def: =0003C4=001D4B}
\end{equation}
is measurable (here~$\varepsilon>0$). Notice that, if $v$ is another
function from $\R^{d}\times\R^{d}$ to $L_{2}\rbr{\cell}$, then
$\tau^{\varepsilon}\rbr{uv}=\rbr{\tau^{\varepsilon}u}\rbr{\tau^{\varepsilon}v}$.
We adopt the notation~$u^{\varepsilon}=\tau^{\varepsilon}u$.

Consider the matrix operator~$\op A^{\varepsilon}\colon H^{1}\rbr{\R^{d}}^{n}\to H^{-1}\rbr{\R^{d}}^{n}$
given~by
\begin{equation}
\op A^{\varepsilon}=D^{*}A^{\varepsilon}D.\label{def: A=001D4B}
\end{equation}
It is easy to see that $\op A^{\varepsilon}$ is bounded, with bound~$C_{\flat}=\norm A_{C}$:
\begin{equation}
\norm{\op A^{\varepsilon}u}_{-1,2,\R^{d}}\le C_{\flat}\norm{Du}_{2,\R^{d}}\label{est: A=001D4B is bounded}
\end{equation}
for~all~$u\in H^{1}\rbr{\R^{d}}^{n}$. Now we impose a condition
that will render $\op A^{\varepsilon}$ elliptic. Namely, we assume
that $\op A^{\varepsilon}$ is coercive uniformly in~$\varepsilon\in\set E$,
where $\set E=(0,\varepsilon_{0}]$ with~$\varepsilon_{0}\in(0,1]$,
that is, there are $c_{A}>0$ and~$C_{A}\ge0$ such that
\begin{equation}
\Re\rbr{A^{\varepsilon}Du,Du}_{\R^{d}}+C_{A}\norm u_{2,\R^{d}}^{2}\ge c_{A}\norm{Du}_{2,\R^{d}}^{2}\label{est: A=001D4B is coercive}
\end{equation}
for~every~$u\in H^{1}\rbr{\R^{d}}^{n}$. It follows that $\op A^{\varepsilon}$
is  $m$\nobreakdash-sectorial with sector
\[
\set S=\bigcbr{z\in\C\colon\abs{\Im z}\le c_{A}^{-1}C_{\flat}\rbr{\Re z+C_{A}}}
\]
independent of~$\varepsilon$. Whenever $\mu\notin\set S$, the
operator~$\op A_{\mu}^{\varepsilon}=\op A^{\varepsilon}-\mu$ is
an isomorphism and hence is invertible; moreover, for any~$f\in H^{-1}\rbr{\R^{d}}^{n}$
we have
\begin{equation}
\norm{\rbr{\op A_{\mu}^{\varepsilon}}^{-1}f}_{1,2,\R^{d}}\lesssim\norm f_{-1,2,\R^{d}}.\label{est: Norm of (A=001D4B-=0003BC)=00207B=0000B9}
\end{equation}

Before proceeding, we make a few remarks about the coercivity condition.
It follows from (\ref{est: A=001D4B is coercive}) (via~Lemma~\ref{lem: Coercivity of D*A(x,=0022C5)D on Q})
that $A$ satisfies the Legendre\textendash Hadamard condition
\begin{equation}
\Re\abr{A\rbr{\wc}\hairspace\xi\otimes\eta,\xi\otimes\eta}\ge c_{A}\abs{\xi}^{2}\abs{\eta}^{2},\qquad\xi\in\R^{d},\eta\in\C^{n},\label{est: Legendre-Hadamard condition}
\end{equation}
so $\op A^{\varepsilon}$ is strongly elliptic for all~$\varepsilon>0$.
The~Legendre\textendash Hadamard condition does not generally imply~(\ref{est: A=001D4B is coercive}).
If~we restrict our attention to the real-valued case, then for
scalar operators the two statements are equivalent. But this is no
longer true  for matrix operators, let alone the complex-valued
case. A~necessary and sufficient algebraic condition on $A$ that
would guarantee~(\ref{est: A=001D4B is coercive}) is not known.

It~is~worthwhile to point out that we have to be able to verify
the coercivity bound for all $\varepsilon$ in some interval~$(0,\varepsilon_{0}]$,
which may be rather difficult. A~sufficient condition not involving~$\varepsilon$
is that the operator~$D^{*}A\rbr{x,\wc}\hairspace D$ is strongly
coercive on $H^{1}\rbr{\R^{d}}^{n}$ and furthermore there is  $c>0$
so that for any~$x\in\R^{d}$ and~$u\in H^{1}\rbr{\R^{d}}^{n}$
\begin{equation}
\Re\rbr{A\rbr{x,\wc}\hairspace Du,Du}_{\R^{d}}\ge c\norm{Du}_{2,\R^{d}}^{2}.\label{est: sufficient condition for A=001D4B coercivity}
\end{equation}
This can be seen by noticing that, by change of variable, the
above inequality remains true with $A\rbr{x,\varepsilon^{-1}y}$
in place of~$A\rbr{x,y}$. Then a partition of unity argument will
do the job, since $A$ is uniformly continuous in the first variable.

As~an~example of $A$ satisfying (\ref{est: sufficient condition for A=001D4B coercivity}),
let $b\rbr D$ be a matrix first-order differential operator with
symbol
\[
\xi\mapsto b\rbr{\xi}=\sum_{k=1}^{d}b_{k}\xi_{k},
\]
where~$b_{k}\in\C^{m\times n}$. Suppose that the symbol has the
property that, for some~$\alpha>0$,
\[
b\rbr{\xi}^{*}b\rbr{\xi}\ge\alpha\abs{\xi}^{2},\qquad\xi\in\R^{d}.
\]
Let $g$ be a function in $C^{0,1}\rbr{\bar{\R}^{d};\tilde L_{\infty}\rbr{\cell}}^{m\times m}$
with $\Re g$ uniformly positive definite. Now if we take $A_{kl}=b_{k}^{*}gb_{l}$,
then application of the Fourier transform  will yield
\[
\begin{aligned}\Re\rbr{A\rbr{x,\wc}\hairspace Du,Du}_{\R^{d}} & =\Re\rbr{g\rbr{x,\wc}\hairspace b\rbr Du,b\rbr Du}_{\R^{d}}\\
 & \ge\alpha\norm{\rbr{\Re g}^{-1/2}}_{C}^{-2}\norm{Du}_{2,\R^{d}}^{2}.
\end{aligned}
\]
Homogenization for self-adjoint operators of this type was studied
by Birman and Suslina in the purely periodic setting (see, e.g.,~\cite{BSu:2001},
\cite{BSu:2003}, \cite{BSu:2005}, \cite{BSu:2006}, \cite{Su:2013-1}
and~\cite{Su:2013-2}) and by Borisov in the locally periodic setting
(see~\cite{Bor:2008}).

Observe that the more restrictive Legendre condition, which amounts
to the uniform positive definiteness of~$\Re A$, does ensure coercivity,
but excludes some strongly elliptic operators with important applications~\textendash{}
such as certain elasticity operators.

\section{\label{sec: Effective operator}Effective operator}

Given $\xi\in\C^{d\times n}$ and~$x\in\R^{d}$, we let $N_{\xi}\rbr{x,\wc}$
be the weak solution of
\begin{equation}
D^{*}A\rbr{x,\wc}\hairspace\rbr{DN_{\xi}\rbr{x,\wc}+\xi}=0\label{def: N}
\end{equation}
in~$\tilde H_{0}^{1}\rbr{\cell}^{n}$. The~function~$N_{\xi}$ is
well defined, since $D^{*}A\rbr{x,\wc}\hairspace\xi$ is a continuous
linear functional on $\tilde H^{1}\rbr{\cell}^{n}$ and the operator~$D^{*}A\rbr{x,\wc}\hairspace D$
is strongly coercive on~$\tilde H^{1}\rbr{\cell}^{n}$, as we shall
now see.
\begin{lem}
\label{lem: Coercivity of D*A(x,=0022C5)D on Q}For~any~$x\in\R^{d}$
and~all~$u\in\tilde H^{1}\rbr{\cell}^{n}$, we have
\begin{equation}
\Re\rbr{A\rbr{x,\wc}\hairspace Du,Du}_{\cell}\ge c_{A}\norm{Du}_{2,\cell}^{2}.\label{est: Coercivity of D*A(x,=0022C5)D on Q}
\end{equation}

\end{lem}
\begin{proof}
Fix $u_{\varepsilon}=\varepsilon u^{\varepsilon}\varphi$ with $u\in\tilde C^{1}\rbr{\cell}^{n}$
and~$\varphi\in C_{c}^{\infty}\rbr{\R^{d}}$. We substitute $u_{\varepsilon}$
into (\ref{est: A=001D4B is coercive}) and let $\varepsilon$ tend
to~0. Then, because $u_{\varepsilon}$ and~$Du_{\varepsilon}-\rbr{Du}^{\varepsilon}\varphi$
converge in $L_{2}$ to~$0$,
\[
\lim_{\varepsilon\to0}\Re\int_{\R^{d}}\abr{A^{\varepsilon}\rbr x\hairspace\rbr{Du}^{\varepsilon}\rbr x,\rbr{Du}^{\varepsilon}\rbr x}\abs{\varphi\rbr x}^{2}\dd x\ge\lim_{\varepsilon\to0}c_{A}\int_{\R^{d}}\abs{\rbr{Du}^{\varepsilon}\rbr x}^{2}\abs{\varphi\rbr x}^{2}\dd x.
\]
It is well known  that if $f\in C_{c}\rbr{\R^{d};\tilde L_{\infty}\rbr{\cell}}$,
then
\[
\lim_{\varepsilon\to0}\int_{\R^{d}}f^{\varepsilon}\rbr x\dd x=\int_{\R^{d}}\int_{\cell}f\rbr{x,y}\dd x\dd y
\]
(see, for instance,~\cite[Lemmas~5.5 and~5.6]{Al:1992}). As~a result,
\[
\Re\int_{\R^{d}}\int_{\cell}\abr{A\rbr{x,y}\hairspace Du\rbr y,Du\rbr y}\abs{\varphi\rbr x}^{2}\dd x\dd y\ge c_{A}\int_{\R^{d}}\int_{\cell}\abs{Du\rbr y}^{2}\abs{\varphi\rbr x}^{2}\dd x\dd y.
\]
Since $\varphi$ is an arbitrary function in $C_{c}^{\infty}\rbr{\R^{d}}$
and since $A$ is continuous in the first variable, we conclude that,
for any~$x\in\R^{d}$,
\[
\Re\int_{\cell}\abr{A\rbr{x,y}\hairspace Du\rbr y,Du\rbr y}\dd y\ge c_{A}\int_{\cell}\abs{Du\rbr y}^{2}\dd y.\qedhere
\]
\end{proof}
It is clear from Lemma~\ref{lem: Coercivity of D*A(x,=0022C5)D on Q}
and Poincar\'{e}'s inequality~(\ref{est: Poincar=0000E9's inequality})
that 
\[
\Re\rbr{A\rbr{x,\wc}\hairspace Du,Du}_{\cell}\gtrsim\norm u_{1,2,\cell}^{2}
\]
for~every~$u\in\tilde H_{0}^{1}\rbr{\cell}^{n}$. Thus, the definition
of $N_{\xi}$ makes good sense.

Denote by $N$ the map sending $\xi$ to~$N_{\xi}$. Evidently,
$N_{\xi}$ depends linearly on $\xi$, so $N$ is simply an operator
of multiplication by a function  (still denoted by~$N$). The~next
lemma shows that $N$ has the same regularity in the first variable
as~$A$.
\begin{rem}
In~what follows, we denote differentiation in the first variable
by $D_{1}$ and differentiation in the second variable by~$D_{2}$.
When no confusion can arise, we omit the subscript and write~$D$,
as we did before.
\end{rem}
\begin{lem}
\label{lem: N is Lipschitz}We have~$N\in C^{0,1}\rbr{\bar{\R}^{d};\tilde H_{0}^{1}\rbr{\cell}}$.
\end{lem}
\begin{proof}
The~identity~(\ref{def: N}), together with Lemma~\ref{lem: Coercivity of D*A(x,=0022C5)D on Q},
yields
\[
c_{A}\norm{DN_{\xi}\rbr{x,\wc}}_{2,\cell}\le\norm{A\rbr{x,\wc}}_{\infty,\cell}\abs{\xi},
\]
whence
\begin{equation}
\norm{D_{2}N}_{L_{\infty}\rbr{\R^{d};L_{2}\rbr{\cell}}}\le c_{A}^{-1}\norm A_{C}.\label{est: L=002082-norm of DN(x,=0000B7)}
\end{equation}
Next, by (\ref{def: N}) again, for any~$x_{1},x_{2}\in\R^{d}$
and~$v\in\tilde H_{0}^{1}\rbr{\cell}^{n}$
\[
\begin{aligned}\hspace{2em} & \hspace{-2em}\bigrbr{A\rbr{x_{2},\wc}\hairspace\rbr{DN_{\xi}\rbr{x_{2},\wc}-DN_{\xi}\rbr{x_{1},\wc}},Dv}_{\cell}\\
 & =-\bigrbr{\rbr{A\rbr{x_{2},\wc}-A\rbr{x_{1},\wc}}\rbr{\xi+DN_{\xi}\rbr{x_{1},\wc}},Dv}_{\cell}.
\end{aligned}
\]
Taking~$v=N_{\xi}\rbr{x_{2},\wc}-N_{\xi}\rbr{x_{1},\wc}$ and using
Lemma~\ref{lem: Coercivity of D*A(x,=0022C5)D on Q}, we obtain
\[
c_{A}\norm{DN_{\xi}\rbr{x_{2},\wc}-DN_{\xi}\rbr{x_{1},\wc}}_{2,\cell}\le\norm{A\rbr{x_{2},\wc}-A\rbr{x_{1},\wc}}_{\infty,\cell}\norm{\xi+DN_{\xi}\rbr{x_{1},\wc}}_{2,\cell}.
\]
It~now follows from (\ref{est: L=002082-norm of DN(x,=0000B7)})
that
\[
\seminorm{D_{2}N}_{C^{0,1}\rbr{\bar{\R}^{d};L_{2}\rbr{\cell}}}\le c_{A}^{-1}\rbr{1+c_{A}^{-1}\norm A_{C}}\seminorm A_{C^{0,1}}.
\]
We have proved that $D_{2}N\in C^{0,1}\rbr{\bar{\R}^{d};L_{2}\rbr{\cell}}$.
But then  Poincar\'{e}'s inequality~(\ref{est: Poincar=0000E9's inequality})
implies that $N\in C^{0,1}\rbr{\bar{\R}^{d};L_{2}\rbr{\cell}}$ as
well.
\end{proof}
Let~$A^{0}\colon\R^{d}\to\B\rbr{\C^{d\times n}}$ be given~by
\begin{equation}
A^{0}\rbr x=\int_{\cell}A\rbr{x,y}\hairspace\rbr{I+D_{2}N\rbr{x,y}}\dd y.\label{def: Coefficient A=002070}
\end{equation}
Since $A$ and~$D_{2}N$ are continuous in the first variable, so
is~$A^{0}$. In~fact, we have $A^{0}\in C^{0,1}\rbr{\bar{\R}^{d}}$.
Indeed, the estimate
\[
\norm{A^{0}}_{C\rbr{\bar{\R}^{d}}}\le\norm A_{C}\norm{I+D_{2}N}_{C}
\]
is immediate from the definition of~$A^{0}$, and that
\[
\seminorm{A^{0}}_{C^{0,1}\rbr{\bar{\R}^{d}}}\le\norm A_{C}\seminorm{D_{2}N}_{C^{0,1}}+\seminorm A_{C^{0,1}}\norm{I+D_{2}N}_{C}
\]
follows by an easy calculation. Hence, $\norm{A^{0}}_{C^{0,1}\rbr{\bar{\R}^{d}}}$
is finite.

Now we define the effective operator~$\op A^{0}\colon H^{1}\rbr{\R^{d}}^{n}\to H^{-1}\rbr{\R^{d}}^{n}$
by setting
\begin{equation}
\op A^{0}=D^{*}A^{0}D.\label{def: A=002070}
\end{equation}
Observe that $\op A^{0}$ is bounded and coercive (recall G\r{a}rding's
inequality) and thus $m$\nobreakdash-sectorial. It can be proved
that $\op A^{0}$ satisfies an estimate similar to (\ref{est: A=001D4B is coercive})
with exactly the same constants, however the bound on its norm may
be different from~(\ref{est: A=001D4B is bounded}). Nevertheless,
the sector for $\op A^{0}$ remains the same as for~$\op A^{\varepsilon}$.
 We briefly sketch the argument; see~\cite[Section~2.3]{Se:2017-2}
for a related proof. First consider the two-scale effective system
as in \cite{Al:1992} and check that the associated form, which is
defined on $H^{1}\rbr{\R^{d}}^{n}\oplus L_{2}\rbr{\R^{d};\tilde H_{0}^{1}\rbr{\cell}}^{n}$~by
\[
u\oplus U\mapsto\rbr{A\rbr{D_{1}u+D_{2}U},D_{1}u+D_{2}U}_{\R^{d}\times\cell},
\]
is $m$\nobreakdash-sectorial with sector~$\set S$. We only remark
that the coercivity is obtained by substituting $u+\varepsilon U^{\varepsilon}$
(with sufficiently smooth $u$ and~$U$) into (\ref{est: A=001D4B is coercive})
for $u$ and letting $\varepsilon$ tend to~$0$; cf. the proof of
Lemma~\ref{lem: Coercivity of D*A(x,=0022C5)D on Q}. Then notice
that
\[
\rbr{\op A^{0}u,u}_{\R^{d}}=\rbr{A\rbr{D_{1}u+D_{2}U},D_{1}u+D_{2}U}_{\R^{d}\times\cell}
\]
provided~$U=ND_{1}u$ (which is definitely in~$L_{2}\rbr{\R^{d};\tilde H_{0}^{1}\rbr{\cell}}^{n}$).
The~claim is proved.

Thus, we see that the operator~$\op A_{\mu}^{0}=\op A^{0}-\mu$
is an isomorphism as long as $\mu$ is outside~$\set S$. In~addition,
standard regularity theory for strongly elliptic systems (see, e.g.,~\cite[Theorem~4.16]{McL:2000})
implies that the pre-image of $L_{2}\rbr{\R^{d}}^{n}$ under $\smash[b]{\op A_{\mu}^{0}}$
is all of $H^{2}\rbr{\R^{d}}^{n}$ and for any~$f\in L_{2}\rbr{\R^{d}}^{n}$
\begin{equation}
\norm{\rbr{\op A_{\mu}^{0}}^{-1}f}_{2,2,\R^{d}}\lesssim\norm f_{2,\R^{d}}.\label{est: Norm of (A=002070-=0003BC)=00207B=0000B9}
\end{equation}

Let us return to our discussion of coercivity at the end of the
previous section.  As~we have seen, (\ref{est: Coercivity of D*A(x,=0022C5)D on Q})
follows from (\ref{est: A=001D4B is coercive}), which in turn is
a consequence of~(\ref{est: sufficient condition for A=001D4B coercivity}).
 On~the other hand,  (\ref{est: Coercivity of D*A(x,=0022C5)D on Q})
does not generally imply~(\ref{est: sufficient condition for A=001D4B coercivity}),
and there are examples (for~$n>1$, of course) where (\ref{est: Coercivity of D*A(x,=0022C5)D on Q})
holds, but (\ref{est: sufficient condition for A=001D4B coercivity})
is false, see~\cite{BF:2015}. In~such cases, a subsequence of
$\smash[b]{\rbr{\op A_{\mu}^{\varepsilon}}^{-1}}$  may still converge
in the weak operator topology to $\smash[b]{\rbr{\op A_{\mu}^{0}}^{-1}}$,
but $\op A^{0}$ will fail to be strongly elliptic, i.e., $A^{0}$
will not satisfy the Legendre\textendash Hadamard condition.

\section{\label{sec: Correctors}Correctors}

Let the operator~$\op K_{\mu}\colon L_{2}\rbr{\R^{d}}^{n}\to\tilde H^{1}\rbr{\R^{d}\times\cell}^{n}$
be given~by
\begin{equation}
\op K_{\mu}=ND_{1}\rbr{\op A_{\mu}^{0}}^{-1}.\label{def: K}
\end{equation}
Lemma~\ref{lem: N is Lipschitz}, combined with the estimate~(\ref{est: Norm of (A=002070-=0003BC)=00207B=0000B9}),
readily implies that $\op K_{\mu}$ is continuous:
\begin{equation}
\norm{\op K_{\mu}f}_{1,2,\R^{d}\times\cell}\lesssim\norm f_{2,\R^{d}}.\label{est: K is bounded}
\end{equation}
The~very same argument shows that $D_{1}D_{2}\op K_{\mu}$ is bounded
on~$L_{2}\rbr{\R^{d}}^{n}$ as well:
\begin{equation}
\norm{D_{1}D_{2}\op K_{\mu}f}_{2,\R^{d}\times\cell}\lesssim\norm f_{2,\R^{d}}.\label{est: D=002081D=002082K is bounded}
\end{equation}

Since we do not impose any extra  assumptions on the coefficients,
the traditional corrector~$\tau^{\varepsilon}\op K_{\mu}$ will
not even map $L_{2}\rbr{\R^{d}}^{n}$ into itself.  So we must first
appropriately regularize the traditional corrector, and a smoothing
operator is used for exactly this purpose.

\subsection{Smoothing}

Let $\op T^{\varepsilon}\colon L_{2}\rbr{\R^{d}\times\cell}\to L_{2}\rbr{\R^{d}\times\cell;L_{2}\rbr{\cell}}$
be the translation operator
\begin{equation}
\op T^{\varepsilon}u\rbr{x,y,z}=u\rbr{x+\varepsilon z,y},\label{def: T=001D4B}
\end{equation}
where $\rbr{x,y}\in\R^{d}\times\cell$ and~$z\in\cell$. Certainly,
for any $u,v\in L_{2}\rbr{\R^{d}\times\cell}$ satisfying $uv\in L_{2}\rbr{\R^{d}\times\cell}$
we have $\op T^{\varepsilon}uv=\rbr{\op T^{\varepsilon}u}\rbr{\op T^{\varepsilon}v}$.
Next, the adjoint of $\op T^{\varepsilon}$ is given~by
\[
\rbr{\op T^{\varepsilon}}^{*}u\rbr{x,y}=\int_{\cell}u\rbr{x-\varepsilon z,y,z}\dd z.
\]
Note that $\rbr{\op T^{\varepsilon}}^{*}$ is defined on $L_{2}\rbr{\R^{d}\times\cell}$
and~$L_{2}\rbr{\R^{d}}$ as well, by way of identifying these spaces
with the corresponding subspaces of~$L_{2}\rbr{\R^{d}\times\cell;L_{2}\rbr{\cell}}$.
We define the Steklov smoothing operator~$\op S^{\varepsilon}\colon L_{2}\rbr{\R^{d}\times\cell}\to L_{2}\rbr{\R^{d}\times\cell}$
to be the restriction of $\rbr{\op T^{\varepsilon}}^{*}$ to~$L_{2}\rbr{\R^{d}\times\cell}$.
In~other words,
\begin{equation}
\op S^{\varepsilon}u\rbr{x,y}=\int_{\cell}\op T^{\varepsilon}u\rbr{x,y,z}\dd z.\label{def: S=001D4B}
\end{equation}
The~operator~$\op S^{\varepsilon}$ is plainly self-adjoint.

Here we collect some facts about $\op T^{\varepsilon}$ and~$\op S^{\varepsilon}$.
\begin{lem}
\label{lem: =0003C4=001D4BT=001D4B}The~restriction of $\tau^{\varepsilon}\op T^{\varepsilon}$
to $\tilde L_{2}\rbr{\R^{d}\times\cell}$ is an isometry.
\end{lem}
\begin{proof}
By~change of variable,
\[
\norm{\tau^{\varepsilon}\op T^{\varepsilon}u}_{2,\R^{d}\times\cell}^{2}=\int_{\R^{d}}\int_{\cell}\abs{u\rbr{x,\varepsilon^{-1}x-z}}^{2}\dd x\dd z.
\]
But since $u$ is periodic in the second variable, this equals~$\norm u_{2,\R^{d}\times\cell}^{2}$.
\end{proof}
A~related result for $\op S^{\varepsilon}$ is the following.
\begin{lem}
\label{lem: =0003C4=001D4BS=001D4B}The~restriction of $\tau^{\varepsilon}\op S^{\varepsilon}$
to $\tilde L_{2}\rbr{\R^{d}\times\cell}$ is bounded, with bound at
most~$1$.
\end{lem}
\begin{proof}
This is immediate from Cauchy's inequality and Lemma~\ref{lem: =0003C4=001D4BT=001D4B}.
\end{proof}
It~is easy to see that both $\op T^{\varepsilon}$ and~$\op S^{\varepsilon}$
converge in the strong operator topology to the identity operator,
yet they do not converge in  norm. The~uniform convergence will,
however, take place if we restrict them to certain Sobolev spaces.
\begin{lem}
\label{lem: T=001D4B-I}For~any~$u\in C_{c}^{\infty}\rbr{\R^{d}\times\cell}$
we have
\begin{equation}
\norm{\rbr{\op T^{\varepsilon}-\op I}u}_{2,\R^{d}\times\cell\times\cell}\lesssim\varepsilon\norm{D_{1}u}_{2,\R^{d}\times\cell}.\label{est: T=001D4B-I is of order =0003B5}
\end{equation}
\end{lem}
\begin{proof}
Notice that
\[
u\rbr{x+\varepsilon z,y}-u\rbr{x,y}=\varepsilon i\int_{0}^{1}\abr{D_{1}u\rbr{x+\varepsilon tz,y},z}\dd t.
\]
Hence,
\[
\norm{\rbr{\op T^{\varepsilon}-\op I}u\rbr{\wc,y,z}}_{2,\R^{d}}\le\varepsilon\cellradius\norm{D_{1}u\rbr{\wc,y}}_{2,\R^{d}},
\]
where $\cellradius=1/2\diam\cell$. Integrating out the~$y$ and~$z$
variables then yields~(\ref{est: T=001D4B-I is of order =0003B5}).
\end{proof}
\begin{lem}
\label{lem: S=001D4B-I}For~any~$u\in C_{c}^{\infty}\rbr{\R^{d}\times\cell}$
we have
\begin{align}
\norm{\rbr{\op S^{\varepsilon}-\op I}u}_{2,\R^{d}\times\cell} & \lesssim\varepsilon\norm{D_{1}u}_{2,\R^{d}\times\cell},\label{est: S=001D4B-I is of order =0003B5}\\
\norm{\rbr{\op S^{\varepsilon}-\op I}u}_{2,\R^{d}\times\cell} & \lesssim\varepsilon^{2}\norm{D_{1}D_{1}u}_{2,\R^{d}\times\cell}.\label{est: S=001D4B-I is of order =0003B5=0000B2}
\end{align}
\end{lem}
\begin{proof}
The~inequality~(\ref{est: S=001D4B-I is of order =0003B5}) comes
from~(\ref{est: T=001D4B-I is of order =0003B5}). To~prove (\ref{est: S=001D4B-I is of order =0003B5=0000B2}),
notice that
\[
u\rbr{x+\varepsilon z,y}-u\rbr{x,y}=\varepsilon i\abr{D_{1}u\rbr{x,y},z}-\varepsilon^{2}\int_{0}^{1}\rbr{1-t}\abr{D_{1}D_{1}u\rbr{x+\varepsilon tz,y}\hairspace z,z}\dd t.
\]
The~first term on the right-hand side has mean value zero for a.e.~$x$
and~$y$ (because $\cell$ is centered at the origin), so
\[
\norm{\rbr{\op S^{\varepsilon}-\op I}u\rbr{\wc,y}}_{2,\R^{d}}\le\varepsilon^{2}\cellradius^{2}\norm{D_{1}D_{1}u\rbr{\wc,y}}_{2,\R^{d}}.
\]
Integrating over $\cell$ completes the proof.
\end{proof}
Now we can prove the following result.
\begin{lem}
\label{lem: =0003C4=001D4BT=001D4B-=0003C4=001D4BS=001D4B}For~any~$u\in\tilde C_{c}^{\infty}\rbr{\R^{d}\times\cell}$
we have
\[
\norm{\tau^{\varepsilon}\op T^{\varepsilon}u-\tau^{\varepsilon}\op S^{\varepsilon}u}_{2,\R^{d}\times\cell}\lesssim\varepsilon\norm{D_{1}u}_{2,\R^{d}\times\cell}.
\]
\end{lem}
\begin{proof}
We write
\[
\tau^{\varepsilon}\op T^{\varepsilon}u-\tau^{\varepsilon}\op S^{\varepsilon}u=\tau^{\varepsilon}\op T^{\varepsilon}\rbr{\op I-\op S^{\varepsilon}}u+\tau^{\varepsilon}\op S^{\varepsilon}\rbr{\op T^{\varepsilon}-\op I}u
\]
(here $\op S^{\varepsilon}\op T^{\varepsilon}$ is understood to
be defined as $\op S^{\varepsilon}\op T^{\varepsilon}=\op T^{\varepsilon}\op S^{\varepsilon}$,
that is, we apply $\op S^{\varepsilon}$ to $\op T^{\varepsilon}u$
regarding the new variable resulting from the operator~$\op T^{\varepsilon}$
as a parameter). Then, it follows from Lemmas~\ref{lem: =0003C4=001D4BT=001D4B}
and~\ref{lem: S=001D4B-I} that
\[
\norm{\tau^{\varepsilon}\op T^{\varepsilon}\rbr{\op I-\op S^{\varepsilon}}u}_{2,\R^{d}\times\cell}\lesssim\varepsilon\norm{D_{1}u}_{2,\R^{d}\times\cell},
\]
while Lemmas~\ref{lem: =0003C4=001D4BS=001D4B} and~\ref{lem: T=001D4B-I}
imply that
\[
\norm{\tau^{\varepsilon}\op S^{\varepsilon}\rbr{\op T^{\varepsilon}-\op I}u}_{2,\R^{d}\times\cell}\lesssim\varepsilon\norm{D_{1}u}_{2,\R^{d}\times\cell}.
\]
These observations combine to give the desired estimate.
\end{proof}
\begin{rem}
\label{rem: Lp variants}We note that the results of Lemmas~\ref{lem: =0003C4=001D4BT=001D4B}\textendash \hairspace\ref{lem: =0003C4=001D4BT=001D4B-=0003C4=001D4BS=001D4B}
persist if we replace the $L_{2}$\nobreakdash-norms by the $L_{p}$\nobreakdash-norms
with~$p\in\sbr{1,\infty}$. This will play a role in what follows.
\end{rem}

\subsection{Correctors}

We define the first corrector~$\op K_{\mu}^{\varepsilon}\colon L_{2}\rbr{\R^{d}}^{n}\to H^{1}\rbr{\R^{d}}^{n}$~by
\begin{equation}
\op K_{\mu}^{\varepsilon}=\tau^{\varepsilon}\op S^{\varepsilon}\op K_{\mu}.\label{def: K=001D4B}
\end{equation}
More explicitly,
\[
\op K_{\mu}^{\varepsilon}f\rbr x=\int_{\cell}N\rbr{x+\varepsilon z,\varepsilon^{-1}x}\hairspace D\rbr{\op A_{\mu}^{0}}^{-1}f\rbr{x+\varepsilon z}\dd z.
\]
Because of the smoothing~$\op S^{\varepsilon}$, this corrector
is bounded with 
\begin{align}
\norm{\op K_{\mu}^{\varepsilon}f}_{2,\R^{d}} & \lesssim\norm f_{2,\R^{d}},\label{est: Norm of K=001D4B}\\
\norm{D\op K_{\mu}^{\varepsilon}f}_{2,\R^{d}} & \lesssim\varepsilon^{-1}\norm f_{2,\R^{d}}.\label{est: Norm of DK=001D4B}
\end{align}
Indeed, using Lemma~\ref{lem: =0003C4=001D4BS=001D4B}, we see
that
\begin{align*}
\norm{\op K_{\mu}^{\varepsilon}f}_{2,\R^{d}} & \le\norm{\op K_{\mu}f}_{2,\R^{d}\times\cell},\\
\norm{D\op K_{\mu}^{\varepsilon}f}_{2,\R^{d}} & \le\norm{D_{1}\op K_{\mu}f}_{2,\R^{d}\times\cell}+\varepsilon^{-1}\norm{D_{2}\op K_{\mu}f}_{2,\R^{d}\times\cell}.
\end{align*}
The~estimates~(\ref{est: Norm of K=001D4B}) and~(\ref{est: Norm of DK=001D4B})
then follow from~(\ref{est: K is bounded}).

While the $L_{2}$\nobreakdash-norm of $\op K_{\mu}^{\varepsilon}f$
is merely uniformly bounded, the $L_{2}$\nobreakdash-norm of $\op S^{\varepsilon}\op K_{\mu}^{\varepsilon}f$
turns out to be of order~$\varepsilon$.
\begin{lem}
\label{lem: S=001D4BK=001D4B}For~any~$\varepsilon\in\set E$ and~$f\in L_{2}\rbr{\R^{d}}^{n}$
we have
\[
\norm{\op S^{\varepsilon}\op K_{\mu}^{\varepsilon}f}_{2,\R^{d}}\lesssim\varepsilon\norm f_{2,\R^{d}}.
\]
\end{lem}
\begin{proof}
By~definition of $\op S^{\varepsilon}$ and~$\op K_{\mu}^{\varepsilon}$,
\[
\op S^{\varepsilon}\op K_{\mu}^{\varepsilon}f\rbr x=\int_{\cell}\int_{\cell}\op T^{\varepsilon}\op K_{\mu}f\rbr{x+\varepsilon w,\varepsilon^{-1}x+z,z}\dd w\dd z.
\]
Since~$\op K_{\mu}f\rbr{x,\wc}$ is periodic and has mean value zero,
we have
\[
\int_{\cell}\int_{\cell}\op K_{\mu}f\rbr{x+\varepsilon w,\varepsilon^{-1}x+z}\dd w\dd z=0,
\]
and~hence
\[
\op S^{\varepsilon}\op K_{\mu}^{\varepsilon}f\rbr x=\int_{\cell}\int_{\cell}\rbr{\op T^{\varepsilon}-\op I}\op K_{\mu}f\rbr{x+\varepsilon w,\varepsilon^{-1}x+z,z}\dd w\dd z.
\]
Changing variables and keeping in mind that $\op K_{\mu}f$ is periodic
in the second variable, we find that
\[
\norm{\op S^{\varepsilon}\op K_{\mu}^{\varepsilon}f}_{2,\R^{d}}\le\norm{\rbr{\op T^{\varepsilon}-\op I}\op K_{\mu}f}_{2,\R^{d}\times\cell\times\cell}.
\]
The~result is therefore immediate from Lemma~\ref{lem: T=001D4B-I}
and the estimate~(\ref{est: K is bounded}).
\end{proof}
To~describe the second corrector, we need some additional notation.
Let  $\rbr{\op A_{\mu}^{\varepsilon}}^{\dual}$ be the  adjoint
of~$\op A_{\mu}^{\varepsilon}$. Then we construct the effective
operator~$\rbr{\op A_{\mu}^{0}}^{\dual}$, the corrector~$\rbr{\op K_{\mu}^{\varepsilon}}^{\dual}$
and  the other objects (which will be marked with ``$+$'' as well)
for $\rbr{\op A_{\mu}^{\varepsilon}}^{\dual}$ just as we did for~$\op A_{\mu}^{\varepsilon}$.
(It may be noted in passing that $\rbr{\op A_{\mu}^{0}}^{\dual}$
is the  adjoint of~$\op A_{\mu}^{0}$.) Of~course, all results
for $\op A_{\mu}^{\varepsilon}$ will transfer to $\rbr{\op A_{\mu}^{\varepsilon}}^{\dual}$.
We shall not explicitly formulate these results here, but  refer
to them by the numbers of the corresponding statements for $\op A_{\mu}^{\varepsilon}$
with ``$+$'' following the reference (for~example, Lemma~\ref{lem: S=001D4BK=001D4B}\dualref\
and the estimate~(\ref{est: Norm of K=001D4B})\dualref).

Define $\op L_{\mu}\colon L_{2}\rbr{\R^{d}}^{n}\to L_{2}\rbr{\R^{d}}^{n}$~by
\begin{equation}
\op L_{\mu}=\rbr{D_{1}\op K_{\mu}^{\dual}}^{*}A\bigrbr{D_{1}\rbr{\op A_{\mu}^{0}}^{-1}+D_{2}\op K_{\mu}}\label{def: L}
\end{equation}
and~$\op M_{\mu}^{\varepsilon}\colon L_{2}\rbr{\R^{d}}^{n}\to L_{2}\rbr{\R^{d}}^{n}$~by
\begin{equation}
\op M_{\mu}^{\varepsilon}=\varepsilon^{-1}\bigrbr{\tau^{\varepsilon}\op T^{\varepsilon}\bigrbr{D_{1}\rbr{\rbr{\op A_{\mu}^{0}}^{\dual}}^{-1}+D_{2}\op K_{\mu}^{\dual}}}^{*}\tau^{\varepsilon}\sbr{A,\op T^{\varepsilon}}\bigrbr{D_{1}\rbr{\op A_{\mu}^{0}}^{-1}+D_{2}\op K_{\mu}}.\label{def: M=001D4B}
\end{equation}
A~more  convenient way of dealing with these operators is to
look at their forms. If we set $u_{0}=\rbr{\op A_{\mu}^{0}}^{-1}f$,
$U=\op K_{\mu}f$ and~$u_{0}^{\dual}=\rbr{\rbr{\op A_{\mu}^{0}}^{\dual}}^{-1}g$,
$U^{\dual}=\op K_{\mu}^{\dual}g$, then
\[
\rbr{\op L_{\mu}f,g}_{\R^{d}}=\bigrbr{A\rbr{D_{1}u_{0}+D_{2}U},D_{1}U^{\dual}}_{\R^{d}\times\cell}
\]
and
\[
\rbr{\op M_{\mu}^{\varepsilon}f,g}_{\R^{d}}=\varepsilon^{-1}\bigrbr{\tau^{\varepsilon}\sbr{A,\op T^{\varepsilon}}\rbr{D_{1}u_{0}+D_{2}U},\tau^{\varepsilon}\op T^{\varepsilon}\rbr{D_{1}u_{0}^{\dual}+D_{2}U^{\dual}}}_{\R^{d}\times\cell}.
\]
Both $\op L_{\mu}$ and~$\op M_{\mu}^{\varepsilon}$ are bounded.
Indeed,
\[
\abs{\rbr{\op L_{\mu}f,g}_{\R^{d}}}\le\norm{A\rbr{D_{1}u_{0}+D_{2}U}}_{2,\R^{d}\times\cell}\norm{D_{1}U^{\dual}}_{2,\R^{d}\times\cell},
\]
and so, according to the estimates~(\ref{est: Norm of (A=002070-=0003BC)=00207B=0000B9}),
(\ref{est: K is bounded}) and~(\ref{est: K is bounded})\dualref,
\begin{equation}
\norm{\op L_{\mu}f}_{2,\R^{d}}\lesssim\norm f_{2,\R^{d}}.\label{est: L}
\end{equation}
Likewise, observing that $\tau^{\varepsilon}\sbr{A,\op T^{\varepsilon}}=\tau^{\varepsilon}\rbr{\op I-\op T^{\varepsilon}}A\cdot\tau^{\varepsilon}\op T^{\varepsilon}$
(by~the~multiplicativity of $\tau^{\varepsilon}$ and~$\op T^{\varepsilon}$),
we conclude that
\[
\abs{\rbr{\op M_{\mu}^{\varepsilon}f,g}_{\R^{d}}}\le\cellradius\seminorm A_{C^{0,1}}\norm{\tau^{\varepsilon}\op T^{\varepsilon}\rbr{D_{1}u_{0}+D_{2}U}}_{2,\R^{d}\times\cell}\norm{\tau^{\varepsilon}\op T^{\varepsilon}\rbr{D_{1}u_{0}^{\dual}+D_{2}U^{\dual}}}_{2,\R^{d}\times\cell}.
\]
This, together with Lemma~\ref{lem: =0003C4=001D4BT=001D4B} and
the estimates~(\ref{est: Norm of (A=002070-=0003BC)=00207B=0000B9}),
(\ref{est: K is bounded}) and (\ref{est: Norm of (A=002070-=0003BC)=00207B=0000B9})\dualref,
(\ref{est: K is bounded})\dualref, yields that
\begin{equation}
\norm{\op M_{\mu}^{\varepsilon}f}_{2,\R^{d}}\lesssim\norm f_{2,\R^{d}}.\label{est: M=001D4B}
\end{equation}

Now we introduce the second corrector~$\op C_{\mu}^{\varepsilon}\colon L_{2}\rbr{\R^{d}}^{n}\to L_{2}\rbr{\R^{d}}^{n}$~by
\begin{equation}
\op C_{\mu}^{\varepsilon}=\rbr{\op K_{\mu}^{\varepsilon}-\op L_{\mu}}-\op M_{\mu}^{\varepsilon}+\rbr{\rbr{\op K_{\mu}^{\varepsilon}}^{\dual}-\op L_{\mu}^{\dual}}^{*}.\label{def: C=001D4B}
\end{equation}
Then (\ref{est: Norm of K=001D4B}), (\ref{est: L}), (\ref{est: M=001D4B})
and (\ref{est: Norm of K=001D4B})\dualref, (\ref{est: L})\dualref\
 imply that $\op C_{\mu}^{\varepsilon}$ is continuous:
\begin{equation}
\norm{\op C_{\mu}^{\varepsilon}f}_{2,\R^{d}}\lesssim\norm f_{2,\R^{d}}.\label{est: C=001D4B}
\end{equation}

\begin{rem}
\label{rem: Can be removed?}From~(\ref{est: M=001D4B}) we know
that the operator norm of $\op M_{\mu}^{\varepsilon}$ is bounded
uniformly in~$\varepsilon$. In~some situations, we can go further
and prove that
\begin{equation}
\norm{\op M_{\mu}^{\varepsilon}f}_{2,\R^{d}}\lesssim\varepsilon\norm f_{2,\R^{d}}.\label{est: M=001D4B, improved}
\end{equation}
The~term~$\op M_{\mu}^{\varepsilon}$ can then be removed from $\op C_{\mu}^{\varepsilon}$,
because, in the context where the corrector~$\op C_{\mu}^{\varepsilon}$
is needed (see~Theorem~\ref{thm: Approximation with 2d corrector}
below), this term  will be absorbed to the error.

The~estimate~(\ref{est: M=001D4B, improved}) is true, for instance,
if $A\in C^{1,1}\rbr{\bar{\R}^{d};\tilde L_{\infty}\rbr{\cell}}$. To~see
this, we notice that if $u,v\in\tilde H^{1}\rbr{\R^{d}\times\cell}^{d\times n}$,
then $uv\in\tilde W_{1}^{1}\rbr{\R^{d}\times\cell}$ and,  by an $L_{1}$\nobreakdash-variant
of Lemma~\ref{lem: =0003C4=001D4BT=001D4B-=0003C4=001D4BS=001D4B}
(see Remark~\ref{rem: Lp variants}),
\[
\begin{aligned}\hspace{2em} & \hspace{-2em}\bigabs{\bigrbr{\tau^{\varepsilon}\rbr{\op I-\op T^{\varepsilon}}A\cdot\tau^{\varepsilon}\op T^{\varepsilon}u,\tau^{\varepsilon}\op T^{\varepsilon}v}_{\R^{d}\times\cell}-\bigrbr{\tau^{\varepsilon}\rbr{\op I-\op S^{\varepsilon}}A\cdot\tau^{\varepsilon}\op T^{\varepsilon}u,\tau^{\varepsilon}\op T^{\varepsilon}v}_{\R^{d}\times\cell}}\\
 & \le\varepsilon\cellradius\seminorm A_{C^{0,1}}\norm{\tau^{\varepsilon}\op T^{\varepsilon}u\bar v-\tau^{\varepsilon}\op S^{\varepsilon}u\bar v}_{1,\R^{d}\times\cell}\lesssim\varepsilon^{2}\norm{D_{1}u\bar v}_{1,\R^{d}\times\cell}
\end{aligned}
\]
(we have reversed the order of integration to pass from $\op S^{\varepsilon}$
to $\op T^{\varepsilon}$ in the second term on the left). This means
that we may replace the function~$\tau^{\varepsilon}\rbr{\op I-\op T^{\varepsilon}}A$
in $\op M_{\mu}^{\varepsilon}$ by $\tau^{\varepsilon}\rbr{\op I-\op S^{\varepsilon}}A$
with error being of order~$\varepsilon$. But since $A\in C^{1,1}\rbr{\bar{\R}^{d};\tilde L_{\infty}\rbr{\cell}}$,
 an $L_{\infty}$\nobreakdash-variant of the second estimate in
Lemma~\ref{lem: S=001D4B-I} (again see Remark~\ref{rem: Lp variants})
will imply that $\op M_{\mu}^{\varepsilon}$ itself is of order~$\varepsilon$.
Another example when (\ref{est: M=001D4B, improved}) holds is the
case where the fast and slow variables  are separated, that is,
$A^{\varepsilon}\rbr x=A\rbr{x_{1},\varepsilon^{-1}x_{2}}$ with~$x=\rbr{x_{1},x_{2}}$.
Since only the rapid oscillations must be regularized, we may choose
$\op T^{\varepsilon}$ to be the translation operator in the  variable~$x_{2}$:
\[
\op T^{\varepsilon}u\rbr{x,y}\rbr{z_{2}}=u\rbr{x_{1},x_{2}+\varepsilon z_{2},y}.
\]
Then $\rbr{\op I-\op T^{\varepsilon}}A$ is  identically zero. Operators
with such coefficients have been studied in~\cite{Se:2017-2}.
\end{rem}
\begin{rem}
\label{rem: Cannot be removed}Given the previous remark, it may
be tempting to conjecture that (\ref{est: M=001D4B, improved}) holds
for all~$A\in C^{0,1}\rbr{\bar{\R}^{d};\tilde L_{\infty}\rbr{\cell}}$.
However, this is not the case, as the following example shows. 
Define
\[
\chi\rbr x=\sum_{k\in\N}k^{-2}\cos2^{k}\pi x.
\]
Then $\chi$ is uniformly continuous, but does not satisfy a H\"{o}lder
condition of any order at all points (see~\cite[Section~4]{H:1916}
for details). Let $A_{1}$ be a uniformly positive definite  Lipschitz
function on $\R$ whose derivative equals $\chi$ on $\rbr{0,1}$
and is $0$ off~$\rbr{0,1}$, and let  $A_{2}\rbr y=4\pi^{1/2}\rbr{2+\sin2\pi y}^{-1}$.
Set~$A\rbr{x,y}=A_{1}\rbr x\hairspace A_{2}\rbr y$. Select an~$f\in L_{2}\rbr{\R}$
in such a way that $\abs{Du_{0}}^{2}=1$ on~$\rbr{0,1}$. It~is
a straightforward, yet tedious, calculation to see that
\[
\rbr{\op M_{\mu}^{\varepsilon_{k}}f,f}_{\R}=\rbr{\log_{2}\varepsilon_{k}^{-1}}^{-2}+O\rbr{\varepsilon_{k}},\qquad k\to\infty,
\]
where~$\varepsilon_{k}=2^{-k}$. In~fact, for any monotone function~$\zeta\in C\rbr{\sbr{0,1}}$
that satisfies $\zeta\rbr 0=0$ and~$\zeta\rbr{\varepsilon}\ge\varepsilon$,
we can construct a uniformly elliptic  operator~$\op A^{\varepsilon}$
on $H^{1}\rbr{\R}$ and find  a sequence~$\cbr{\varepsilon_{k}}_{k\in\N}$
converging to $0$ such that
\[
\rbr{\op M_{\mu}^{\varepsilon_{k}}f,f}_{\R}=\zeta\rbr{\varepsilon_{k}}+O\rbr{\varepsilon_{k}},\qquad k\to\infty,
\]
for~some~$f\in L_{2}\rbr{\R}$. The~idea is to  adjust gaps in
the Fourier series for~$\chi$.
\end{rem}
\begin{rem}
We observe that $\op L_{\mu}$ can be written in the form
\[
\op L_{\mu}=\rbr{\op A_{\mu}^{0}}^{-1}D^{*}\op LD\rbr{\op A_{\mu}^{0}}^{-1},
\]
where~$\op L\colon H^{1}\rbr{\R^{d}}^{n}\to L_{2}\rbr{\R^{d}}^{n}$
is a first-order differential operator with bounded coefficients:
\[
\op L=\int_{\cell}N^{\dual}\rbr{\wc,y}^{*}D_{1}^{*}A\rbr{\wc,y}\hairspace\rbr{I+D_{2}N\rbr{\wc,y}}\dd y
\]
(cf.~\cite[Remark~2.6]{Se:2017-2}). Likewise, we can write $\op M_{\mu}^{\varepsilon}$~as
\[
\op M_{\mu}^{\varepsilon}=\rbr{\op A_{\mu}^{0}}^{-1}D^{*}M_{\varepsilon}D\rbr{\op A_{\mu}^{0}}^{-1}
\]
where~$M_{\varepsilon}$ is the bounded function given~by
\[
M_{\varepsilon}\rbr x=\varepsilon^{-1}\int_{\cell}\rbr{I+D_{2}N^{\dual}\rbr{x,\varepsilon^{-1}x+z}}^{*}\findiff_{\varepsilon z}A\rbr{x,\varepsilon^{-1}x+z}\hairspace\rbr{I+D_{2}N\rbr{x,\varepsilon^{-1}x+z}}\dd z
\]
with~$\findiff_{\varepsilon z}A\rbr{x,y}=A\rbr{x+\varepsilon z,y}-A\rbr{x,y}$.
\end{rem}

\section{\label{sec: Main results}Main results}

Now we formulate the main results of the  paper.
\begin{thm}
\label{thm: Convergence and Approximation with 1st corrector}If~$\mu\notin\set S$,
then for any $\varepsilon\in\set E$ and~$f\in L_{2}\rbr{\R^{d}}^{n}$
we have 
\begin{align}
\norm{\rbr{\op A_{\mu}^{\varepsilon}}^{-1}f-\rbr{\op A_{\mu}^{0}}^{-1}f}_{2,\R^{d}} & \lesssim\varepsilon\norm f_{2,\R^{d}},\label{est: Convergence}\\
\norm{D\rbr{\op A_{\mu}^{\varepsilon}}^{-1}f-D\rbr{\op A_{\mu}^{0}}^{-1}f-\varepsilon D\op K_{\mu}^{\varepsilon}f}_{2,\R^{d}} & \lesssim\varepsilon\norm f_{2,\R^{d}}.\label{est: Approximation with 1st corrector}
\end{align}
The~estimates are sharp with respect to the order, and the constants
depend only on the parameters~$d$, $n$, $\mu$, the norm~$\norm A_{C^{0,1}}$
and the constants~$c_{A}$ and~$C_{A}$ in the coercivity bound.
\end{thm}
\begin{thm}
\label{thm: Approximation with 2d corrector}If~$\mu\notin\set S$,
then for any $\varepsilon\in\set E$ and~$f\in L_{2}\rbr{\R^{d}}^{n}$
it holds that 
\begin{equation}
\norm{\rbr{\op A_{\mu}^{\varepsilon}}^{-1}f-\rbr{\op A_{\mu}^{0}}^{-1}f-\varepsilon\op C_{\mu}^{\varepsilon}f}_{2,\R^{d}}\lesssim\varepsilon^{2}\norm f_{2,\R^{d}}.\label{est: Approximation with 2d corrector}
\end{equation}
The~estimate is sharp with respect to the order, and the constant
depends only on the parameters~$d$, $n$, $\mu$, the norm~$\norm A_{C^{0,1}}$
and the constants~$c_{A}$ and~$C_{A}$ in the coercivity bound.
\end{thm}

\begin{rem}
\label{rem: Correctors in periodic case}These results should
be compared with those in~\cite{Se:2017-2}. Suppose that $A^{\varepsilon}$
is periodic, that is, $A^{\varepsilon}\rbr x=A\rbr{x_{1},\varepsilon^{-1}x_{2}}$,
where~$x=\rbr{x_{1},x_{2}}$. In~\cite{Se:2017-2} we  proved estimates
similar to (\ref{est: Convergence})\textendash (\ref{est: Approximation with 2d corrector}),
but with different correctors in (\ref{est: Approximation with 1st corrector})
and~(\ref{est: Approximation with 2d corrector}). The~difference
stems from the smoothing operator. As~mentioned earlier, in the
periodic case we may reduce $\op T^{\varepsilon}$ to the translation
operator in the  variable~$x_{2}$. Then $\op S^{\varepsilon}$
 will involve averaging over~$\varepsilon\cell$, with $\cell$
being the basic cell for the lattice of periods (not necessarily
of full rank).  The~Gelfand transform provides another smoothing
that is, in a sense, dual to the first one and  involves averaging
over the dual cell~$\varepsilon^{-1}\cell^{*}$ in the reciprocal
space. (Here $\cell^{*}$ is the Wigner\textendash Seitz cell in
the dual lattice.) It~is this last smoothing that appeared in~\cite{Se:2017-2}.
One can verify directly that either of these may be used in the
corrector~$\op K_{\mu}^{\varepsilon}$.  As~for~$\op L_{\mu}$
and~$\op M_{\mu}^{\varepsilon}$, the former does not depend on smoothing
and is just the same as in~\cite{Se:2017-2}, and the latter is
zero by the choice of~$\op T^{\varepsilon}$ (see~Remark~\ref{rem: Can be removed?}).
\end{rem}

\begin{rem}
Theorems~\ref{thm: Convergence and Approximation with 1st corrector}
and~\ref{thm: Approximation with 2d corrector} can be extended
to allow  all $\mu\notin\spec\op A^{0}$, though  it
may be necessary to replace $\set E$ by a smaller set~$\set E_{\mu}$
depending on~$\mu$. Indeed, the proofs of the theorems go over
without change to the case $\mu\notin\spec\op A^{0}$ provided we
establish estimates similar to (\ref{est: Norm of (A=001D4B-=0003BC)=00207B=0000B9})
and~(\ref{est: Norm of (A=002070-=0003BC)=00207B=0000B9}). By~the~first
resolvent identity, this amounts to checking  that $\op A_{\mu}^{\varepsilon}$
 as an operator on $L_{2}\rbr{\R^{d}}^{n}$ has a uniformly bounded
inverse. Suppose that $\mu\in\set S$ (otherwise~$\set E_{\mu}=\set E$).
We know from Theorem~\ref{thm: Convergence and Approximation with 1st corrector}
that if $\nu\notin\set S$, then
\[
\norm{\rbr{\op A_{\nu}^{\varepsilon}}^{-1}f-\rbr{\op A_{\nu}^{0}}^{-1}f}_{2,\R^{d}}\le C_{\nu}\varepsilon\norm f_{2,\R^{d}}
\]
for~all~$\varepsilon\in\set E$ and~$f\in L_{2}\rbr{\R^{d}}^{n}$.
Therefore, using the identity
\[
\begin{aligned}\rbr{\op A_{\mu}^{\varepsilon}}^{-1}-\rbr{\op A_{\mu}^{0}}^{-1} & =\Bigrbr{\op I-\rbr{\mu-\nu}\op A_{\nu}^{0}\rbr{\op A_{\mu}^{0}}^{-1}\bigrbr{\rbr{\op A_{\nu}^{\varepsilon}}^{-1}-\rbr{\op A_{\nu}^{0}}^{-1}}}^{-1}\\
 & \quad\quad\times\op A_{\nu}^{0}\rbr{\op A_{\mu}^{0}}^{-1}\bigrbr{\rbr{\op A_{\nu}^{\varepsilon}}^{-1}-\rbr{\op A_{\nu}^{0}}^{-1}}\op A_{\nu}^{0}\rbr{\op A_{\mu}^{0}}^{-1},
\end{aligned}
\]
we see that $\rbr{\op A_{\mu}^{\varepsilon}}^{-1}$ is bounded on
$L_{2}\rbr{\R^{d}}^{n}$ uniformly in~$\varepsilon\le\varepsilon_{\mu,\nu}\meet\varepsilon_{0}$,
where
\[
\varepsilon_{\mu,\nu}<\frac{\dist\rbr{\mu,\spec\op A^{0}}}{C_{\nu}\abs{\mu-\nu}\bigrbr{\dist\rbr{\mu,\spec\op A^{0}}+\abs{\mu-\nu}}}.
\]
It~follows that we can set~$\set E_{\mu}=(0,\varepsilon_{\mu,\nu}\meet\varepsilon_{0}]$.
\end{rem}

\begin{rem}
We note that the operator~$D\smash[b]{\rbr{\op A_{\mu}^{\varepsilon}}^{-1}}$
converges in the uniform topology if and only if $D_{2}^{*}A\rbr{x,\wc}\hairspace\xi=0$
on $\tilde H^{1}\rbr{\cell}^{n}$ for every $x\in\R^{d}$ and~$\xi\in\C^{d\times n}$,
in which case $N$ is zero and hence so is~$\op K_{\mu}^{\varepsilon}$.
Notice also that the effective coefficients are then obtained by
ordinary averaging over~$\cell$.
\end{rem}
\begin{rem}
\label{rem: Holder coefficients}By~keeping track of $\seminorm A_{C^{0,1}}$
in estimates, we can find that the constants on the right of (\ref{est: Convergence})
and~(\ref{est: Approximation with 1st corrector}) depend linearly
on~$\seminorm A_{C^{0,1}}$, while the constant on the right of
(\ref{est: Approximation with 2d corrector}), quadratically. These
 observations  play a role in proving results similar to Theorems~\ref{thm: Convergence and Approximation with 1st corrector}
and~\ref{thm: Approximation with 2d corrector}   when the coefficients
are  H\"{o}lder continuous, or even continuous, in the slow variable.
The~key idea is to use mollification  to replace $A$ with a 
function~$A_{\delta}$ that is Lipschitz in the first variable.
In~the case of H\"{o}lder continuous coefficients, we are able to
control  both the convergence rate of $A_{\delta}$ to $A$ in a
H\"{o}lder seminorm and the growth rate of~$\seminorm{A_{\delta}}_{C^{0,1}}$
in terms of $\delta$ as~$\delta\to0$. In~the end, this allows
us to obtain the desired operator  estimates. However, if the coefficients
are only continuous, such an approach yields the convergence of the
resolvent, but not the rate. These results have been announced in~\cite{Se:2017-3};
detailed proofs will appear elsewhere.
\end{rem}

\section{\label{sec: Proof of the main results}Proof of the  main results}

Our first task is to obtain an identity involving  $\rbr{\op A_{\mu}^{\varepsilon}}^{-1}$,
$\rbr{\op A_{\mu}^{0}}^{-1}$ and~$\op K_{\mu}^{\varepsilon}$ that
will play a crucial role in the proofs.

Fix~$f\in L_{2}\rbr{\R^{d}}^{n}$ and~$g\in H^{-1}\rbr{\R^{d}}^{n}$.
Let $u_{0}=\rbr{\op A_{\mu}^{0}}^{-1}f$, $U=\op K_{\mu}f$, $U_{\varepsilon}=\op K_{\mu}^{\varepsilon}f$
and~$u_{\varepsilon}^{\dual}=\rbr{\rbr{\op A_{\mu}^{\varepsilon}}^{\dual}}^{-1}g$.
Then we have
\begin{equation}
\begin{aligned}\bigrbr{\rbr{\op A_{\mu}^{\varepsilon}}^{-1}f-\rbr{\op A_{\mu}^{0}}^{-1}f-\varepsilon\op K_{\mu}^{\varepsilon}f,g}_{\R^{d}} & =\rbr{\op A^{0}u_{0},u_{\varepsilon}^{\dual}}_{\R^{d}}-\rbr{\op A^{\varepsilon}\rbr{\op S^{\varepsilon}u_{0}+\varepsilon U_{\varepsilon}},u_{\varepsilon}^{\dual}}_{\R^{d}}\\
 & \quad-\rbr{\op A^{\varepsilon}\rbr{\op I-\op S^{\varepsilon}}u_{0},u_{\varepsilon}^{\dual}}_{\R^{d}}+\varepsilon\mu\rbr{U_{\varepsilon},u_{\varepsilon}^{\dual}}_{\R^{d}}.
\end{aligned}
\label{eq: First step}
\end{equation}
Let us look at the first two terms on the right. By~the~definition
of the effective coefficients,
\[
\rbr{\op A^{0}u_{0},u_{\varepsilon}^{\dual}}_{\R^{d}}=\bigrbr{A\rbr{D_{1}u_{0}+D_{2}U},D_{1}u_{\varepsilon}^{\dual}}_{\R^{d}\times\cell}.
\]
Then Lemma~\ref{lem: =0003C4=001D4BT=001D4B} yields that
\begin{equation}
\rbr{\op A^{0}u_{0},u_{\varepsilon}^{\dual}}_{\R^{d}}=\bigrbr{\tau^{\varepsilon}\op T^{\varepsilon}A\rbr{D_{1}u_{0}+D_{2}U},\op T^{\varepsilon}D_{1}u_{\varepsilon}^{\dual}}_{\R^{d}\times\cell}\label{eq: A=002070}
\end{equation}
(notice here that $u_{\varepsilon}^{\dual}$ does not depend on the
second variable). On~the other hand,
\begin{equation}
\begin{aligned}\rbr{\op A^{\varepsilon}\rbr{\op S^{\varepsilon}u_{0}+\varepsilon U_{\varepsilon}},u_{\varepsilon}^{\dual}}_{\R^{d}} & =\bigrbr{\tau^{\varepsilon}A\op T^{\varepsilon}\rbr{D_{1}u_{0}+D_{2}U},D_{1}u_{\varepsilon}^{\dual}}_{\R^{d}\times\cell}\\
 & \quad+\varepsilon\bigrbr{\tau^{\varepsilon}A\op T^{\varepsilon}D_{1}U,D_{1}u_{\varepsilon}^{\dual}}_{\R^{d}\times\cell}.
\end{aligned}
\label{eq: A=001D4B}
\end{equation}
Commuting $\op T^{\varepsilon}$ past $A$ in the first term on the
right and combining the resulting identity with (\ref{eq: A=002070}),
we conclude that
\begin{equation}
\begin{aligned}\hspace{2em} & \hspace{-2em}\rbr{\op A^{0}u_{0},u_{\varepsilon}^{\dual}}_{\R^{d}}-\rbr{\op A^{\varepsilon}\rbr{\op S^{\varepsilon}u_{0}+\varepsilon U_{\varepsilon}},u_{\varepsilon}^{\dual}}_{\R^{d}}\\
 & =\bigrbr{\tau^{\varepsilon}\op T^{\varepsilon}A\rbr{D_{1}u_{0}+D_{2}U},\rbr{\op T^{\varepsilon}-\op I}D_{1}u_{\varepsilon}^{\dual}}_{\R^{d}\times\cell}\\
 & \quad-\bigrbr{\tau^{\varepsilon}\sbr{A,\op T^{\varepsilon}}\rbr{D_{1}u_{0}+D_{2}U},D_{1}u_{\varepsilon}^{\dual}}_{\R^{d}\times\cell}\\
 & \quad-\varepsilon\bigrbr{\tau^{\varepsilon}A\op T^{\varepsilon}D_{1}U,D_{1}u_{\varepsilon}^{\dual}}_{\R^{d}\times\cell}.
\end{aligned}
\label{eq: A=002070 - A=001D4B, in the beginning}
\end{equation}

We would like to be able to prove that the norm of the operator
corresponding to the left-hand side is of order~$\varepsilon$. It
is clear from the previous discussion that the last two terms on the
right satisfy the desired estimate. The~same would be true for the
first term if we could integrate by parts and transfer $D_{1}$
from $\rbr{\op T^{\varepsilon}-\op I}u_{\varepsilon}^{\dual}$ to~$A\rbr{D_{1}u_{0}+D_{2}U}$.
The following technical result will be useful for this purpose.
\begin{lem}
\label{lem: Differentiation of =0003C4=0003B5T=0003B5F}Let $F\in C^{0,1}\rbr{\bar{\R}^{d};\tilde L_{2}\rbr{\cell}}^{d\times n}$
be such that $D_{2}^{*}F\rbr{x,\wc}=0$ on $\tilde H^{1}\rbr{\cell}^{n}$
for each~$x\in\R^{d}$. Then $D_{1}^{*}\tau^{\varepsilon}\op T^{\varepsilon}F=\tau^{\varepsilon}\op T^{\varepsilon}D_{1}^{*}F$
on $C_{c}^{1}\rbr{\R^{d};C\rbr{\cell}}^{n}$ for any~$\varepsilon>0$.
\end{lem}
\begin{proof}
It suffices to check the assertion for $\varepsilon=1$, because the
general result will then follow from this special case applied to
the function~$\rbr{x,y}\mapsto F\rbr{\varepsilon x,y}$. After
a change of variables, we must show that, for any $\varphi\in C_{c}^{1}\rbr{\R^{d};C\rbr{\cell}}^{n}$,
\begin{equation}
\int_{\R^{d}}\int_{\cell}\abr{F\rbr{x,x+y},D_{1}\varphi\rbr{x,y}}\dd x\dd y=\int_{\R^{d}}\int_{\cell}\abr{D_{1}^{*}F\rbr{x,x+y},\varphi\rbr{x,y}}\dd x\dd y.\label{eq: Differentiation of =0003C4=0003B5T=0003B5F | Change of variables}
\end{equation}
Were $F$  smooth, this would be nothing but the usual integration
by parts formula. But we can find a sequence of divergence free smooth
functions that converges, in a certain sense, to the function~$F$,
which will yield the desired conclusion.

If~$e_{k}\rbr y=e^{2\pi i\abr{y,k}}$, where~$k\in\Z^{d}$,
then we let $F_{K}\rbr{x,\wc}$ denote the  partial sum of the Fourier
series for~$F\rbr{x,\wc}$:
\[
F_{K}\rbr{x,\wc}=\sum_{\abs k\le K}\hat{F}_{k}\rbr x\hairspace e_{k}.
\]
By~hypothesis, $D_{2}^{*}F\rbr{x,\wc}=0$ on~$\tilde H^{1}\rbr{\cell}^{n}$,
so
\[
\abr{\hat{F}_{k}\rbr x,k}=\rbr{2\pi}^{-1}\int_{\cell}\abr{F\rbr{x,y},De_{k}\rbr y}\dd y=0
\]
for~each~$k\in\Z^{d}$. An~integration by parts then gives
\begin{equation}
\int_{\R^{d}}\int_{\cell}\abr{F_{K}\rbr{x,x+y},D_{1}\varphi\rbr{x,y}}\dd x\dd y=\int_{\R^{d}}\int_{\cell}\abr{D_{1}^{*}F_{K}\rbr{x,x+y},\varphi\rbr{x,y}}\dd x\dd y\label{eq: Differentiation of =0003C4=0003B5T=0003B5F | Identity for Fk}
\end{equation}
(notice here that $D\hat{F}_{k}\rbr x$ are exactly the Fourier coefficients
of $D_{1}F\rbr{x,\wc}$).

Our goal now is to pass from (\ref{eq: Differentiation of =0003C4=0003B5T=0003B5F | Identity for Fk})
to~(\ref{eq: Differentiation of =0003C4=0003B5T=0003B5F | Change of variables}).
Let $f$ be a function in $C^{0,1}\rbr{\bar{\R}^{d};\tilde L_{2}\rbr{\cell}}$
and let $f_{K}\rbr{x,\wc}$ be the partial sum of the Fourier series
for~$f\rbr{x,\wc}$. We claim that $f_{K}\to f$ in the weak\nobreakdash-$*$
topology on $C_{c}\rbr{\R^{d}\times\cell}^{*}$ as~$K\to\infty$.
Indeed, given any $\psi\in C_{c}\rbr{\R^{d}\times\cell}$, the sequence
of  functions~$x\mapsto\rbr{f_{K}\rbr{x,\wc},\psi\rbr{x,\wc}}_{\cell}$
converges pointwise to the function~$x\mapsto\rbr{f\rbr{x,\wc},\psi\rbr{x,\wc}}_{\cell}$,
because $f_{K}\rbr{x,\wc}\to f\rbr{x,\wc}$ in~$L_{2}$. In~addition,
all the functions in the sequence are supported in a compact set
and are uniformly bounded, since
\[
\abs{\rbr{f_{K}\rbr{x,\wc},\psi\rbr{x,\wc}}_{\cell}}\le\norm{f\rbr{x,\wc}}_{2,\cell}\norm{\psi\rbr{x,\wc}}_{2,\cell}\le\norm f_{C}\norm{\psi}_{C}.
\]
We see that $\rbr{f_{K},\psi}_{\R^{d}\times\cell}\to\rbr{f,\psi}_{\R^{d}\times\cell}$
by the Lebesgue dominated convergence theorem, and the claim follows.

The proof is completed now by letting $K\to\infty$ in~(\ref{eq: Differentiation of =0003C4=0003B5T=0003B5F | Identity for Fk}).
\end{proof}

By~definition, we have $A\rbr{D_{1}u_{0}+D_{2}U}=A\rbr{I+D_{2}N}Du_{0}$.
Assume for the moment that~$u_{0},u_{\varepsilon}^{\dual}\in C_{c}^{\infty}\rbr{\R^{d}}^{n}$.
We recall that, for each $x\in\R^{d}$ and~$\xi\in\C^{d\times n}$,
$D_{2}^{*}A\rbr{x,\wc}\hairspace\rbr{I+D_{2}N\rbr{x,\wc}}\xi=0$ on
$\tilde H^{1}\rbr{\cell}^{n}$, so Lemma~\ref{lem: Differentiation of =0003C4=0003B5T=0003B5F}
applies to show that
\begin{equation}
\begin{aligned}\hspace{2em} & \hspace{-2em}\bigrbr{\tau^{\varepsilon}\op T^{\varepsilon}A\rbr{D_{1}u_{0}+D_{2}U},\rbr{\op T^{\varepsilon}-\op I}D_{1}u_{\varepsilon}^{\dual}}_{\R^{d}\times\cell}\\
 & =\bigrbr{\tau^{\varepsilon}\op T^{\varepsilon}D_{1}^{*}A\rbr{D_{1}u_{0}+D_{2}U},\rbr{\op T^{\varepsilon}-\op I}u_{\varepsilon}^{\dual}}_{\R^{d}\times\cell}
\end{aligned}
\label{eq: First term, once Lemma applied}
\end{equation}
for~every~$u_{0},u_{\varepsilon}^{\dual}\in C_{c}^{\infty}\rbr{\R^{d}}^{n}$.
Moreover, since the form
\[
\rbr{u_{0},u_{\varepsilon}^{\dual}}\mapsto\bigrbr{\tau^{\varepsilon}\op T^{\varepsilon}A\rbr{D_{1}u_{0}+D_{2}U},\rbr{\op T^{\varepsilon}-\op I}D_{1}u_{\varepsilon}^{\dual}}_{\R^{d}\times\cell}
\]
is  continuous on~$H^{1}\rbr{\R^{d}}^{n}\times H^{1}\rbr{\R^{d}}^{n}$
and since the form
\[
\rbr{u_{0},u_{\varepsilon}^{\dual}}\mapsto\bigrbr{\tau^{\varepsilon}\op T^{\varepsilon}D_{1}^{*}A\rbr{D_{1}u_{0}+D_{2}U},\rbr{\op T^{\varepsilon}-\op I}u_{\varepsilon}^{\dual}}_{\R^{d}\times\cell}
\]
is continuous on~$H^{2}\rbr{\R^{d}}^{n}\times L_{2}\rbr{\R^{d}}^{n}$,
the last equality holds  for any $u_{0}\in H^{2}\rbr{\R^{d}}^{n}$
and~$u_{\varepsilon}^{\dual}\in H^{1}\rbr{\R^{d}}^{n}$.

Now that we have this result, (\ref{eq: A=002070 - A=001D4B, in the beginning})
becomes
\begin{equation}
\begin{aligned}\hspace{2em} & \hspace{-2em}\rbr{\op A^{0}u_{0},u_{\varepsilon}^{\dual}}_{\R^{d}}-\rbr{\op A^{\varepsilon}\rbr{\op S^{\varepsilon}u_{0}+\varepsilon U_{\varepsilon}},u_{\varepsilon}^{\dual}}_{\R^{d}}\\
 & =\bigrbr{\tau^{\varepsilon}\op T^{\varepsilon}D_{1}^{*}A\rbr{D_{1}u_{0}+D_{2}U},\rbr{\op T^{\varepsilon}-\op I}u_{\varepsilon}^{\dual}}_{\R^{d}\times\cell}\\
 & \quad-\bigrbr{\tau^{\varepsilon}\sbr{A,\op T^{\varepsilon}}\rbr{D_{1}u_{0}+D_{2}U},D_{1}u_{\varepsilon}^{\dual}}_{\R^{d}\times\cell}\\
 & \quad-\varepsilon\bigrbr{\tau^{\varepsilon}A\op T^{\varepsilon}D_{1}U,D_{1}u_{\varepsilon}^{\dual}}_{\R^{d}\times\cell}.
\end{aligned}
\label{eq: A=002070 - A=001D4B}
\end{equation}
Putting (\ref{eq: A=002070 - A=001D4B}) into (\ref{eq: First step}),
we finally obtain the desired identity:
\begin{equation}
\begin{aligned}\hspace{2em} & \hspace{-2em}\bigrbr{\rbr{\op A_{\mu}^{\varepsilon}}^{-1}f-\rbr{\op A_{\mu}^{0}}^{-1}f-\varepsilon\op K_{\mu}^{\varepsilon}f,g}_{\R^{d}}\\
 & =\bigrbr{\tau^{\varepsilon}\op T^{\varepsilon}D_{1}^{*}A\rbr{D_{1}u_{0}+D_{2}U},\rbr{\op T^{\varepsilon}-\op I}u_{\varepsilon}^{\dual}}_{\R^{d}\times\cell}\\
 & \quad-\bigrbr{\tau^{\varepsilon}\sbr{A,\op T^{\varepsilon}}\rbr{D_{1}u_{0}+D_{2}U},D_{1}u_{\varepsilon}^{\dual}}_{\R^{d}\times\cell}\\
 & \quad-\varepsilon\bigrbr{\tau^{\varepsilon}A\op T^{\varepsilon}D_{1}U,D_{1}u_{\varepsilon}^{\dual}}_{\R^{d}\times\cell}\\
 & \quad-\rbr{\op A^{\varepsilon}\rbr{\op I-\op S^{\varepsilon}}u_{0},u_{\varepsilon}^{\dual}}_{\R^{d}}+\varepsilon\mu\rbr{U_{\varepsilon},u_{\varepsilon}^{\dual}}_{\R^{d}}.
\end{aligned}
\label{eq: The identity}
\end{equation}

We are now in a position  to prove the theorems.
\begin{proof}[Proof of Theorem~\textup{\ref{thm: Convergence and Approximation with 1st corrector}}]
We estimate each term in~(\ref{eq: The identity}). By~Lemmas~\ref{lem: =0003C4=001D4BT=001D4B}
and~\ref{lem: T=001D4B-I},
\begin{equation}
\begin{aligned}\hspace{2em} & \hspace{-2em}\bigabs{\bigrbr{\tau^{\varepsilon}\op T^{\varepsilon}D_{1}^{*}A\rbr{D_{1}u_{0}+D_{2}U},\rbr{\op T^{\varepsilon}-\op I}u_{\varepsilon}^{\dual}}_{\R^{d}\times\cell}}\\
 & \le\norm{\tau^{\varepsilon}\op T^{\varepsilon}D_{1}^{*}A\rbr{D_{1}u_{0}+D_{2}U}}_{2,\R^{d}\times\cell}\norm{\rbr{\op T^{\varepsilon}-\op I}u_{\varepsilon}^{\dual}}_{2,\R^{d}\times\cell}\\
 & \lesssim\varepsilon\bigrbr{\norm{Du_{0}}_{1,2,\R^{d}}+\norm{D_{1}D_{2}U}_{2,\R^{d}\times\cell}+\norm{D_{2}U}_{2,\R^{d}\times\cell}}\norm{Du_{\varepsilon}^{\dual}}_{2,\R^{d}}.
\end{aligned}
\label{est: Proof Thm1 | Est 1}
\end{equation}
Using Lemma~\ref{lem: =0003C4=001D4BT=001D4B} again, we see that
\begin{equation}
\begin{aligned}\hspace{2em} & \hspace{-2em}\bigabs{\bigrbr{\tau^{\varepsilon}\sbr{A,\op T^{\varepsilon}}\rbr{D_{1}u_{0}+D_{2}U},D_{1}u_{\varepsilon}^{\dual}}_{\R^{d}\times\cell}}\\
 & \le\varepsilon\cellradius\seminorm A_{C^{0,1}}\norm{\tau^{\varepsilon}\op T^{\varepsilon}\rbr{D_{1}u_{0}+D_{2}U}}_{2,\R^{d}\times\cell}\norm{D_{1}u_{\varepsilon}^{\dual}}_{2,\R^{d}\times\cell}\\
 & \lesssim\varepsilon\bigrbr{\norm{Du_{0}}_{2,\R^{d}}+\norm{D_{2}U}_{2,\R^{d}\times\cell}}\norm{Du_{\varepsilon}^{\dual}}_{2,\R^{d}}
\end{aligned}
\label{est: Proof Thm1 | Est 2}
\end{equation}
(recall that $\cellradius=1/2\diam\cell$) and
\begin{equation}
\begin{aligned}\varepsilon\bigabs{\bigrbr{\tau^{\varepsilon}A\op T^{\varepsilon}D_{1}U,D_{1}u_{\varepsilon}^{\dual}}_{\R^{d}\times\cell}} & \le\varepsilon\norm A_{C}\norm{\tau^{\varepsilon}\op T^{\varepsilon}D_{1}U}_{2,\R^{d}\times\cell}\norm{D_{1}u_{\varepsilon}^{\dual}}_{2,\R^{d}\times\cell}\\
 & \lesssim\varepsilon\norm{D_{1}U}_{2,\R^{d}\times\cell}\norm{Du_{\varepsilon}^{\dual}}_{2,\R^{d}}.
\end{aligned}
\label{est: Proof Thm1 | Est 3}
\end{equation}
Next, it follows from the estimate~(\ref{est: A=001D4B is bounded})
and Lemma~\ref{lem: S=001D4B-I} that
\[
\begin{aligned}\abs{\rbr{\op A^{\varepsilon}\rbr{\op I-\op S^{\varepsilon}}u_{0},u_{\varepsilon}^{\dual}}_{\R^{d}}} & \lesssim\norm{\rbr{\op I-\op S^{\varepsilon}}u_{0}}_{1,2,\R^{d}}\norm{u_{\varepsilon}^{\dual}}_{1,2,\R^{d}}\\
 & \lesssim\varepsilon\norm{Du_{0}}_{1,2,\R^{d}}\norm{u_{\varepsilon}^{\dual}}_{1,2,\R^{d}}.
\end{aligned}
\]
Finally, Lemma~\ref{lem: =0003C4=001D4BS=001D4B} yields
\[
\varepsilon\abs{\rbr{U_{\varepsilon},u_{\varepsilon}^{\dual}}_{\R^{d}}}\le\varepsilon\norm{U_{\varepsilon}}_{2,\R^{d}}\norm{u_{\varepsilon}^{\dual}}_{2,\R^{d}}\lesssim\varepsilon\norm U_{2,\R^{d}\times\cell}\norm{u_{\varepsilon}^{\dual}}_{2,\R^{d}}.
\]
In~summary, we have found that
\[
\begin{aligned}\hspace{2em} & \hspace{-2em}\bigabs{\bigrbr{\rbr{\op A_{\mu}^{\varepsilon}}^{-1}f-\rbr{\op A_{\mu}^{0}}^{-1}f-\varepsilon\op K_{\mu}^{\varepsilon}f,g}_{\R^{d}}}\\
 & \lesssim\varepsilon\bigrbr{\norm{Du_{0}}_{1,2,\R^{d}}+\norm{D_{1}D_{2}U}_{2,\R^{d}\times\cell}+\norm U_{1,2,\R^{d}\times\cell}}\norm{u_{\varepsilon}^{\dual}}_{1,2,\R^{d}}.
\end{aligned}
\]

Now suppose that $g\in L_{2}\rbr{\R^{d}}^{n}$. Then  from (\ref{est: Norm of (A=002070-=0003BC)=00207B=0000B9}),
(\ref{est: K is bounded}), (\ref{est: D=002081D=002082K is bounded}),
(\ref{est: Norm of K=001D4B}) and~(\ref{est: Norm of (A=001D4B-=0003BC)=00207B=0000B9})\dualref,
\[
\bigabs{\bigrbr{\rbr{\op A_{\mu}^{\varepsilon}}^{-1}f-\rbr{\op A_{\mu}^{0}}^{-1}f,g}_{\R^{d}}}\lesssim\varepsilon\norm f_{2,\R^{d}}\norm g_{2,\R^{d}},
\]
which proves~(\ref{est: Convergence}). On~the~other hand, setting
$g=D^{*}h$ where $h\in L_{2}\rbr{\R^{d}}^{d\times n}$ and using
(\ref{est: Norm of (A=002070-=0003BC)=00207B=0000B9}), (\ref{est: K is bounded}),
(\ref{est: D=002081D=002082K is bounded}) and~(\ref{est: Norm of (A=001D4B-=0003BC)=00207B=0000B9})\dualref,
we obtain
\[
\bigabs{\bigrbr{\rbr{\op A_{\mu}^{\varepsilon}}^{-1}f-\rbr{\op A_{\mu}^{0}}^{-1}f-\varepsilon\op K_{\mu}^{\varepsilon}f,D^{*}h}_{\R^{d}}}\lesssim\varepsilon\norm f_{2,\R^{d}}\norm h_{2,\R^{d}},
\]
which proves~(\ref{est: Approximation with 1st corrector}).
\end{proof}
\begin{proof}[Proof of Theorem~\textup{\ref{thm: Approximation with 2d corrector}}]
Let $u_{0}^{\dual}=\rbr{\rbr{\op A_{\mu}^{0}}^{\dual}}^{-1}g$,
$U^{\dual}=\op K_{\mu}^{\dual}g$ and~$U_{\varepsilon}^{\dual}=\rbr{\op K_{\mu}^{\varepsilon}}^{\dual}g$.
As~a~first step, we rewrite the corrector~$\op C_{\mu}^{\varepsilon}$
dropping, as we may, terms with operator norm of order~$\varepsilon$.

By~the very definition of~$\op C_{\mu}^{\varepsilon}$,
\[
\rbr{\op C_{\mu}^{\varepsilon}f,g}_{\R^{d}}=\rbr{\op K_{\mu}^{\varepsilon}f,g}_{\R^{d}}-\rbr{\op L_{\mu}f,g}_{\R^{d}}-\rbr{\op M_{\mu}^{\varepsilon}f,g}_{\R^{d}}+\rbr{f,\rbr{\op K_{\mu}^{\varepsilon}}^{\dual}g}_{\R^{d}}-\rbr{f,\op L_{\mu}^{\dual}g}_{\R^{d}}.
\]
We  claim that
\begin{equation}
\begin{aligned}\hspace{2em} & \hspace{-2em}-\varepsilon\rbr{\op L_{\mu}f,g}_{\R^{d}}-\varepsilon\rbr{\op M_{\mu}^{\varepsilon}f,g}_{\R^{d}}+\varepsilon\rbr{f,\rbr{\op K_{\mu}^{\varepsilon}}^{\dual}g}_{\R^{d}}-\varepsilon\rbr{f,\op L_{\mu}^{\dual}g}_{\R^{d}}\\
 & \approx\bigrbr{\tau^{\varepsilon}\op T^{\varepsilon}D_{1}^{*}A\rbr{D_{1}u_{0}+D_{2}U},\rbr{\op T^{\varepsilon}-\op I}\rbr{u_{0}^{\dual}+\varepsilon U_{\varepsilon}^{\dual}}}_{\R^{d}\times\cell}\\
 & \quad-\bigrbr{\tau^{\varepsilon}\sbr{A,\op T^{\varepsilon}}\rbr{D_{1}u_{0}+D_{2}U},D_{1}\rbr{u_{0}^{\dual}+\varepsilon U_{\varepsilon}^{\dual}}}_{\R^{d}\times\cell}\\
 & \quad-\varepsilon\bigrbr{\tau^{\varepsilon}A\op T^{\varepsilon}D_{1}U,D_{1}\rbr{u_{0}^{\dual}+\varepsilon U_{\varepsilon}^{\dual}}}_{\R^{d}\times\cell},
\end{aligned}
\label{eq: Proof Thm2 | C=0003B5 approx.}
\end{equation}
where the symbol~$\approx$ is used to indicate equality up to
terms that will eventually be absorbed into the error.

Indeed, by Lemma~\ref{lem: =0003C4=001D4BT=001D4B} we have
\[
\rbr{\op L_{\mu}f,g}_{\R^{d}}=\bigrbr{\tau^{\varepsilon}\op T^{\varepsilon}D_{1}^{*}A\rbr{D_{1}u_{0}+D_{2}U},\tau^{\varepsilon}\op T^{\varepsilon}U^{\dual}}_{\R^{d}\times\cell}.
\]
Now observe that  $\tau^{\varepsilon}\op T^{\varepsilon}U^{\dual}$
may be replaced by $\tau^{\varepsilon}\op S^{\varepsilon}U^{\dual}$.
This is so because
\[
\begin{aligned}\hspace{2em} & \hspace{-2em}\bigabs{\bigrbr{\tau^{\varepsilon}\op T^{\varepsilon}D_{1}^{*}A\rbr{D_{1}u_{0}+D_{2}U},\tau^{\varepsilon}\op T^{\varepsilon}U^{\dual}-\tau^{\varepsilon}\op S^{\varepsilon}U^{\dual}}_{\R^{d}\times\cell}}\\
 & \le\norm{\tau^{\varepsilon}\op T^{\varepsilon}D_{1}^{*}A\rbr{D_{1}u_{0}+D_{2}U}}_{2,\R^{d}\times\cell}\norm{\tau^{\varepsilon}\op T^{\varepsilon}U^{\dual}-\tau^{\varepsilon}\op S^{\varepsilon}U^{\dual}}_{2,\R^{d}\times\cell},
\end{aligned}
\]
whence, by Lemmas~\ref{lem: =0003C4=001D4BT=001D4B} and~\ref{lem: =0003C4=001D4BT=001D4B-=0003C4=001D4BS=001D4B}
and the estimates~(\ref{est: Norm of (A=002070-=0003BC)=00207B=0000B9}),
(\ref{est: K is bounded}), (\ref{est: D=002081D=002082K is bounded})
and~(\ref{est: K is bounded})\dualref,
\[
\bigabs{\bigrbr{\tau^{\varepsilon}\op T^{\varepsilon}D_{1}^{*}A\rbr{D_{1}u_{0}+D_{2}U},\tau^{\varepsilon}\op T^{\varepsilon}U^{\dual}-\tau^{\varepsilon}\op S^{\varepsilon}U^{\dual}}_{\R^{d}\times\cell}}\lesssim\varepsilon\norm f_{2,\R^{d}}\norm g_{2,\R^{d}}.
\]
Recalling that $U_{\varepsilon}^{\dual}=\tau^{\varepsilon}\op S^{\varepsilon}U^{\dual}$,
we see that
\begin{equation}
\rbr{\op L_{\mu}f,g}_{\R^{d}}\approx\bigrbr{\tau^{\varepsilon}\op T^{\varepsilon}D_{1}^{*}A\rbr{D_{1}u_{0}+D_{2}U},U_{\varepsilon}^{\dual}}_{\R^{d}\times\cell}.\label{eq: Proof Thm2 | L=0003BC}
\end{equation}

We next want to show that
\begin{equation}
\rbr{f,\op L_{\mu}^{\dual}g}_{\R^{d}}\approx\bigrbr{\tau^{\varepsilon}A\op T^{\varepsilon}D_{1}U,D_{1}\rbr{u_{0}^{\dual}+\varepsilon U_{\varepsilon}^{\dual}}}_{\R^{d}\times\cell}.\label{eq: Proof Thm2 | L=0003BC=00207A}
\end{equation}
According to Lemma~\ref{lem: =0003C4=001D4BT=001D4B},
\[
\rbr{f,\op L_{\mu}^{\dual}g}_{\R^{d}}=\bigrbr{\tau^{\varepsilon}\op T^{\varepsilon}AD_{1}U,\tau^{\varepsilon}\op T^{\varepsilon}\rbr{D_{1}u_{0}^{\dual}+D_{2}U^{\dual}}}_{\R^{d}\times\cell}.
\]
We commute $\op T^{\varepsilon}$ through $A$ and use Lemma~\ref{lem: =0003C4=001D4BT=001D4B}
and the estimates~(\ref{est: K is bounded}) and (\ref{est: Norm of (A=002070-=0003BC)=00207B=0000B9})\dualref,~(\ref{est: K is bounded})\dualref\
to get
\[
\rbr{f,\op L_{\mu}^{\dual}g}_{\R^{d}}\approx\bigrbr{\tau^{\varepsilon}A\op T^{\varepsilon}D_{1}U,\tau^{\varepsilon}\op T^{\varepsilon}\rbr{D_{1}u_{0}^{\dual}+D_{2}U^{\dual}}}_{\R^{d}\times\cell}
\]
(notice here that $\tau^{\varepsilon}\sbr{A,\op T^{\varepsilon}}=\tau^{\varepsilon}\rbr{\op I-\op T^{\varepsilon}}A\cdot\tau^{\varepsilon}\op T^{\varepsilon}$).
A~similar argument using Lemma~\ref{lem: T=001D4B-I} shows that
 $\tau^{\varepsilon}\op T^{\varepsilon}D_{1}u_{0}^{\dual}$ (which
is, of course, equal to~$\op T^{\varepsilon}D_{1}u_{0}^{\dual}$)
may be replaced by~$D_{1}u_{0}^{\dual}$. With~a~little extra care
we can  pass from~$\tau^{\varepsilon}\op T^{\varepsilon}D_{2}U^{\dual}$
to $\varepsilon D_{1}U_{\varepsilon}^{\dual}$, as well. Indeed, $\varepsilon D_{1}U_{\varepsilon}^{\dual}=\varepsilon\tau^{\varepsilon}\op S^{\varepsilon}D_{1}U^{\dual}+\tau^{\varepsilon}\op S^{\varepsilon}D_{2}U^{\dual}$,
where $\varepsilon\tau^{\varepsilon}\op S^{\varepsilon}D_{1}U^{\dual}$
creates another error term and $\tau^{\varepsilon}\op S^{\varepsilon}D_{2}U^{\dual}$
is handled exactly as above, by Lemma~\ref{lem: =0003C4=001D4BT=001D4B-=0003C4=001D4BS=001D4B}.
Hence (\ref{eq: Proof Thm2 | L=0003BC=00207A}) is proved.

Repeating these last arguments for $\tau^{\varepsilon}\op T^{\varepsilon}\rbr{D_{1}u_{0}^{\dual}+D_{2}U^{\dual}}$,
we find also that
\begin{equation}
\rbr{\op M_{\mu}^{\varepsilon}f,g}_{\R^{d}}\approx\varepsilon^{-1}\bigrbr{\tau^{\varepsilon}\sbr{A,\op T^{\varepsilon}}\rbr{D_{1}u_{0}+D_{2}U},D_{1}\rbr{u_{0}^{\dual}+\varepsilon U_{\varepsilon}^{\dual}}}_{\R^{d}\times\cell}.\label{eq: Proof Thm2 | M=0003B5}
\end{equation}

Let us turn to the  term involving~$\rbr{\op K_{\mu}^{\varepsilon}}^{\dual}$.
By~the~definition of~$u_{0}$ and~$U_{\varepsilon}^{\dual}$,
\[
\rbr{f,\rbr{\op K_{\mu}^{\varepsilon}}^{\dual}g}_{\R^{d}}=\rbr{\op A^{0}u_{0},U_{\varepsilon}^{\dual}}_{\R^{d}}-\mu\rbr{u_{0},U_{\varepsilon}^{\dual}}_{\R^{d}}.
\]
Applying Lemmas~\ref{lem: S=001D4B-I} and~\ref{lem: S=001D4BK=001D4B}\dualref\
and the estimates~(\ref{est: Norm of (A=002070-=0003BC)=00207B=0000B9})
and~(\ref{est: Norm of K=001D4B})\dualref\ yields
\[
\abs{\rbr{u_{0},U_{\varepsilon}^{\dual}}_{\R^{d}}}\le\abs{\rbr{\rbr{\op S^{\varepsilon}-\op I}u_{0},U_{\varepsilon}^{\dual}}_{\R^{d}}}+\abs{\rbr{u_{0},\op S^{\varepsilon}U_{\varepsilon}^{\dual}}_{\R^{d}}}\lesssim\varepsilon\norm f_{2,\R^{d}}\norm g_{2,\R^{d}},
\]
so
\[
\rbr{f,\rbr{\op K_{\mu}^{\varepsilon}}^{\dual}g}_{\R^{d}}\approx\rbr{\op A^{0}u_{0},U_{\varepsilon}^{\dual}}_{\R^{d}}.
\]
Thus, from Lemma~\ref{lem: =0003C4=001D4BT=001D4B} and the definition
of the effective coefficients, we have
\begin{equation}
\rbr{f,\rbr{\op K_{\mu}^{\varepsilon}}^{\dual}g}_{\R^{d}}\approx\bigrbr{\tau^{\varepsilon}\op T^{\varepsilon}D_{1}^{*}A\rbr{D_{1}u_{0}+D_{2}U},\op T^{\varepsilon}U_{\varepsilon}^{\dual}}_{\R^{d}\times\cell}.\label{eq: Proof Thm2 | (K=0003BC=0003B5)=00207A}
\end{equation}

To~summarize: by (\ref{eq: Proof Thm2 | L=0003BC})\textendash (\ref{eq: Proof Thm2 | (K=0003BC=0003B5)=00207A}),
(\ref{eq: Proof Thm2 | C=0003B5 approx.}) reduces to showing that
\begin{equation}
\bigabs{\bigrbr{\tau^{\varepsilon}\op T^{\varepsilon}D_{1}^{*}A\rbr{D_{1}u_{0}+D_{2}U},\rbr{\op T^{\varepsilon}-\op I}u_{0}^{\dual}}_{\R^{d}\times\cell}}\lesssim\varepsilon^{2}\norm f_{2,\R^{d}}\norm g_{2,\R^{d}}.\label{est: Proof Thm2 | Remaining term in C=0003B5}
\end{equation}
Let us prove~(\ref{est: Proof Thm2 | Remaining term in C=0003B5}).
From~Lemma~\ref{lem: Differentiation of =0003C4=0003B5T=0003B5F},
we know that
\[
\begin{aligned}\hspace{2em} & \hspace{-2em}\bigrbr{\tau^{\varepsilon}\op T^{\varepsilon}D_{1}^{*}A\rbr{D_{1}u_{0}+D_{2}U},\rbr{\op T^{\varepsilon}-\op I}u_{0}^{\dual}}_{\R^{d}\times\cell}\\
 & =\bigrbr{\tau^{\varepsilon}\op T^{\varepsilon}A\rbr{D_{1}u_{0}+D_{2}U},\rbr{\op T^{\varepsilon}-\op I}D_{1}u_{0}^{\dual}}_{\R^{d}\times\cell}
\end{aligned}
\]
(cf.~(\ref{eq: First term, once Lemma applied})). Lemmas~\ref{lem: T=001D4B-I}
and~\ref{lem: =0003C4=001D4BT=001D4B-=0003C4=001D4BS=001D4B} and
the estimates~(\ref{est: Norm of (A=002070-=0003BC)=00207B=0000B9}),
(\ref{est: K is bounded}), (\ref{est: D=002081D=002082K is bounded})
and~(\ref{est: Norm of (A=002070-=0003BC)=00207B=0000B9})\dualref\
enable us to replace $\tau^{\varepsilon}\op T^{\varepsilon}A\rbr{D_{1}u_{0}+D_{2}U}$
with $\tau^{\varepsilon}\op S^{\varepsilon}A\rbr{D_{1}u_{0}+D_{2}U}$.
Reversing the order of integration to switch $\op T^{\varepsilon}$
and~$\op S^{\varepsilon}$ and again using Lemma~\ref{lem: Differentiation of =0003C4=0003B5T=0003B5F},
we get
\[
\begin{aligned}\hspace{2em} & \hspace{-2em}\bigrbr{\tau^{\varepsilon}\op T^{\varepsilon}D_{1}^{*}A\rbr{D_{1}u_{0}+D_{2}U},\rbr{\op T^{\varepsilon}-\op I}u_{0}^{\dual}}_{\R^{d}\times\cell}\\
 & \approx\bigrbr{\tau^{\varepsilon}\op T^{\varepsilon}D_{1}^{*}A\rbr{D_{1}u_{0}+D_{2}U},\rbr{\op S^{\varepsilon}-\op I}u_{0}^{\dual}}_{\R^{d}\times\cell}.
\end{aligned}
\]
It~then follows from Lemmas~\ref{lem: =0003C4=001D4BT=001D4B} and~\ref{lem: S=001D4B-I}
and the estimates~(\ref{est: Norm of (A=002070-=0003BC)=00207B=0000B9}),
(\ref{est: K is bounded}), (\ref{est: D=002081D=002082K is bounded})
and~(\ref{est: Norm of (A=002070-=0003BC)=00207B=0000B9})\dualref\
that
\[
\bigabs{\bigrbr{\tau^{\varepsilon}\op T^{\varepsilon}D_{1}^{*}A\rbr{D_{1}u_{0}+D_{2}U},\rbr{\op S^{\varepsilon}-\op I}u_{0}^{\dual}}_{\R^{d}\times\cell}}\lesssim\varepsilon^{2}\norm f_{2,\R^{d}}\norm g_{2,\R^{d}}.
\]
We have verified (\ref{est: Proof Thm2 | Remaining term in C=0003B5}),
and therefore the claim is established.

Now we subtract (\ref{eq: Proof Thm2 | C=0003B5 approx.}) from
(\ref{eq: The identity}) to obtain
\[
\begin{aligned}\hspace{2em} & \hspace{-2em}\bigrbr{\rbr{\op A_{\mu}^{\varepsilon}}^{-1}f-\rbr{\op A_{\mu}^{0}}^{-1}f-\varepsilon\op C_{\mu}^{\varepsilon}f,g}_{\R^{d}}\\
 & \approx\bigrbr{\tau^{\varepsilon}\op T^{\varepsilon}D_{1}^{*}A\rbr{D_{1}u_{0}+D_{2}U},\rbr{\op T^{\varepsilon}-\op I}\rbr{u_{\varepsilon}^{\dual}-u_{0}^{\dual}-\varepsilon U_{\varepsilon}^{\dual}}}_{\R^{d}\times\cell}\\
 & \quad-\bigrbr{\tau^{\varepsilon}\sbr{A,\op T^{\varepsilon}}\rbr{D_{1}u_{0}+D_{2}U},D_{1}\rbr{u_{\varepsilon}^{\dual}-u_{0}^{\dual}-\varepsilon U_{\varepsilon}^{\dual}}}_{\R^{d}\times\cell}\\
 & \quad-\varepsilon\bigrbr{\tau^{\varepsilon}A\op T^{\varepsilon}D_{1}U,D_{1}\rbr{u_{\varepsilon}^{\dual}-u_{0}^{\dual}-\varepsilon U_{\varepsilon}^{\dual}}}_{\R^{d}\times\cell}\\
 & \quad-\rbr{\op A^{\varepsilon}\rbr{\op I-\op S^{\varepsilon}}u_{0},u_{\varepsilon}^{\dual}}_{\R^{d}}+\varepsilon\mu\rbr{U_{\varepsilon},u_{\varepsilon}^{\dual}}_{\R^{d}}.
\end{aligned}
\]
Using the inequalities~(\ref{est: Proof Thm1 | Est 1}), (\ref{est: Proof Thm1 | Est 2})
and~(\ref{est: Proof Thm1 | Est 3}) with $u_{\varepsilon}^{\dual}-u_{0}^{\dual}-\varepsilon U_{\varepsilon}^{\dual}$
in place of $u_{\varepsilon}^{\dual}$ and then applying the estimates~(\ref{est: Norm of (A=002070-=0003BC)=00207B=0000B9}),
(\ref{est: K is bounded}), (\ref{est: D=002081D=002082K is bounded})
and~(\ref{est: Approximation with 1st corrector})\dualref, we see
that the norms of the operators associated with the first three forms
on the right are of order~$\varepsilon^{2}$. As~for~the~last
two forms, we write
\[
\rbr{\op A^{\varepsilon}\rbr{\op I-\op S^{\varepsilon}}u_{0},u_{\varepsilon}^{\dual}}_{\R^{d}}=\bigrbr{\rbr{\op I-\op S^{\varepsilon}}u_{0},g+\bar{\mu}u_{\varepsilon}^{\dual}}_{\R^{d}}
\]
and
\[
\varepsilon\rbr{U_{\varepsilon},u_{\varepsilon}^{\dual}}_{\R^{d}}=\varepsilon\rbr{\op S^{\varepsilon}U_{\varepsilon},u_{\varepsilon}^{\dual}}_{\R^{d}}+\varepsilon\rbr{U_{\varepsilon},\rbr{\op I-\op S^{\varepsilon}}u_{\varepsilon}^{\dual}}_{\R^{d}}.
\]
Then, by Lemma~\ref{lem: S=001D4B-I} and the estimates~(\ref{est: Norm of (A=002070-=0003BC)=00207B=0000B9})
and~(\ref{est: Norm of (A=001D4B-=0003BC)=00207B=0000B9})\dualref,
\[
\begin{aligned}\abs{\rbr{\op A^{\varepsilon}\rbr{\op I-\op S^{\varepsilon}}u_{0},u_{\varepsilon}^{\dual}}_{\R^{d}}} & \le\norm{\rbr{\op I-\op S^{\varepsilon}}u_{0}}_{2,\R^{d}}\bigrbr{\norm g_{2,\R^{d}}+\abs{\mu}\norm{u_{\varepsilon}^{\dual}}_{2,\R^{d}}}\\
 & \lesssim\varepsilon^{2}\norm f_{2,\R^{d}}\norm g_{2,\R^{d}},
\end{aligned}
\]
while, by Lemmas~\ref{lem: S=001D4B-I} and~\ref{lem: S=001D4BK=001D4B}
and the estimates~(\ref{est: Norm of K=001D4B}) and~(\ref{est: Norm of (A=001D4B-=0003BC)=00207B=0000B9})\dualref,
\[
\begin{aligned}\varepsilon\abs{\rbr{U_{\varepsilon},u_{\varepsilon}^{\dual}}_{\R^{d}}} & \le\varepsilon\norm{\op S^{\varepsilon}U_{\varepsilon}}_{2,\R^{d}}\norm{u_{\varepsilon}^{\dual}}_{2,\R^{d}}+\varepsilon\norm{U_{\varepsilon}}_{2,\R^{d}}\norm{\rbr{\op I-\op S^{\varepsilon}}u_{\varepsilon}^{\dual}}_{2,\R^{d}}\\
 & \lesssim\varepsilon^{2}\norm f_{2,\R^{d}}\norm g_{2,\R^{d}}.
\end{aligned}
\]
The~proof is complete.
\end{proof}

\section*{Acknowledgment}

The~author is grateful to T.~A.~Suslina for helpful discussions.

\bibliographystyle{amsalpha}
\bibliography{bibliography}

\begin{thebibliography}{OShY90}

\bibitem[A92]{Al:1992}
\textsc{G.~Allaire},
\newblock \textit{Homogenization and two-scale convergence},
\newblock SIAM J. Math. Anal., 23 (1992), pp.~1482--1518.

\bibitem[BP84]{BP:1984}
\textsc{N.~Bakhvalov and G.~Panasenko},
\newblock \textit{Homogenisation: Averaging Processes in Periodic Media:
                  Mathematical Problems in the Mechanics of Composite Materials},
\newblock Nauka, Moscow, 1984 (in Russian);
\newblock Kluwer Academic, Dordrecht, 1989 (in English).

\bibitem[BLP78]{BLP:1978}
\textsc{A.~Bensoussan, J.-L.~Lions and G.~Papanicolaou},
\newblock \textit{Asymptotic Analysis for Periodic Structures},
\newblock North-Holland, Amsterdam, 1978.

\bibitem[BSu01]{BSu:2001}
\textsc{M.~Sh.~Birman and T.~A.~Suslina},
\newblock \textit{Threshold effects near the lower edge of the spectrum
                  for periodic differential operators of mathematical physics},
\newblock in Systems, Approximation, Singular Integral Operators, and Related Topics,
          A.~A.~Borichev and N.~K.~Nikolski, eds.,
\newblock Birkh\"auser, Basel, 2001, pp.~71--107.

\bibitem[BSu03]{BSu:2003}
\bysame,
\newblock \textit{Second order periodic differential operators. Threshold properties and homogenization},
\newblock Algebra i~Analiz, 15 (2003), no.~5, pp.~1--108 (in Russian);
\newblock St.~Petersburg Math.~J., 15 (2004), pp.~639--714 (in English).

\bibitem[BSu05]{BSu:2005}
\bysame,
\newblock \textit{Homogenization with corrector term for periodic elliptic differential operators},
\newblock Algebra i~Analiz, 17 (2005), no.~6, pp.~1--104 (in Russian);
\newblock St.~Petersburg Math.~J., 17 (2006), pp.~897--973 (in English).

\bibitem[BSu06]{BSu:2006}
\bysame,
\newblock \textit{Homogenization with corrector for periodic differential operators.
                  Approximation of solutions in the Sobolev class~$H^{1}(\mathbb{R}^{d})$},
\newblock Algebra i~Analiz, 18 (2006), no.~6, pp.~1--130 (in Russian);
\newblock St.~Petersburg Math.~J., 18 (2007), pp.~857--955 (in English).

\bibitem[B08]{Bor:2008}
\textsc{D.~I.~Borisov},
\newblock \textit{Asymptotics for the solutions of elliptic systems with rapidly oscillating coefficients},
\newblock Algebra i~Analiz, 20 (2008), no.~2, pp.~19--42 (in Russian);
\newblock St.~Petersburg Math.~J., 20 (2009), pp.~175--191 (in English).

\bibitem[BF15]{BF:2015}
\textsc{M.~Briane and G.~A.~Francfort},
\newblock \textit{Loss of ellipticity through homogenization in linear elasticity},
\newblock Math. Models Methods Appl. Sci., 25 (2015), pp.~905--928.

\bibitem[ChC16]{ChC:2016}
\textsc{K.~D.~Cherednichenko and S.~Cooper},
\newblock \textit{Resolvent estimates for high-contrast elliptic problems with periodic coefficients},
\newblock Arch. Ration. Mech. Anal., 219 (2016), pp.~1061--1086.

\bibitem[Gri02]{Gri:2002}
\textsc{G.~Griso},
\newblock \textit{Estimation d'erreur et \'{e}clatement en homog\'{e}n\'{e}isation p\'{e}riodique},
\newblock C.~R.~Math. Acad. Sci. Paris, 335 (2002), pp.~333--336.

\bibitem[Gri04]{Gri:2004}
\bysame,
\newblock \textit{Error estimate and unfolding for periodic homogenization},
\newblock Asymptot. Anal., 40 (2004), pp.~269--286.

\bibitem[Gri06]{Gri:2006}
\bysame,
\newblock \textit{Interior error estimate for periodic homogenization},
\newblock Anal. Appl., 4 (2006), pp.~61--79.

\bibitem[H16]{H:1916}
\textsc{G.~H.~Hardy},
\newblock \textit{Weierstrass's non-differentiable function},
\newblock Trans. Amer. Math. Soc., 17 (1916), pp.~301--325.

\bibitem[KLS12]{KLS:2012}
\textsc{C.~E.~Kenig, F.~Lin and Z.~Shen},
\newblock \textit{Convergence rates in $L_{2}$ for elliptic homogenization problems},
\newblock Arch. Ration. Mech. Anal., 203 (2012), pp.~1009--1036.

\bibitem[McL00]{McL:2000}
\textsc{W.~McLean},
\newblock \textit{Strongly Elliptic Systems and Boundary Integral Equations},
\newblock Cambridge University Press, Cambridge, 2000.

\bibitem[PT07]{PasT:2007}
\textsc{S.~E.~Pastukhova and R.~N.~Tikhomirov},
\newblock \textit{Operator estimates in reiterated and locally periodic homogenization},
\newblock Dokl. Acad. Nauk, 415 (2007), pp.~304--309 (in Russian);
\newblock Dokl. Math., 76 (2007), pp.~548--553 (in English).

\bibitem[Se17\textsubscript{1}]{Se:2017-2}
\textsc{N.~N.~Senik},
\newblock \textit{Homogenization for non-self-adjoint periodic elliptic
                  operators on an infinite cylinder},
\newblock SIAM J. Math.~Anal., 49 (2017), pp.~874--898.

\bibitem[Se17\textsubscript{2}]{Se:2017-3}
\bysame,
\newblock \textit{On homogenization for non-self-adjoint locally periodic
                  elliptic operators},
\newblock Funktsional. Anal. i Prilozhen., 51 (2017), no.~2, pp.~92--96 (in Russian);
\newblock Funct. Anal. Appl., 51 (2017), no.~2, to appear (in English).

\bibitem[Su13\textsubscript{1}]{Su:2013-1}
\textsc{T.~A.~Suslina},
\newblock \textit{Homogenization of the Dirichlet problem for elliptic systems:
                  $L^2$-operator error estimates},
\newblock Mathematika, 59 (2013), pp.~463--476.

\bibitem[Su13\textsubscript{2}]{Su:2013-2}
\bysame,
\newblock \textit{Homogenization of the Neumann problem for elliptic systems
                  with periodic coefficients},
\newblock SIAM J. Math. Anal., 45 (2013), pp.~3453--3493.

\bibitem[ZhKO93]{ZhKO:1993}
\textsc{V.~V.~Zhikov, S.~M.~Kozlov and O.~A.~Oleinik},
\newblock \textit{Homogenization of Differential Operators and Integral Functionals},
\newblock Nauka, Moscow, 1993 (in Russian);
\newblock Springer, Berlin, 1994 (in English).

\bibitem[Zh05]{Zh:2005}
\textsc{V.~V.~Zhikov},
\newblock \textit{On operator estimates in homogenization theory},
\newblock Dokl. Acad. Nauk, 403 (2005), pp.~305--308 (in Russian);
\newblock Dokl. Math., 72 (2005), pp.~535--538 (in English).

\bibitem[ZhP05]{ZhPas:2005}
\textsc{V.~V.~Zhikov and S.~E.~Pastukhova},
\newblock \textit{On operator estimates for some problems in homogenization theory},
\newblock Russ.~J. Math. Phys., 12 (2005), pp.~515--524.

\bibitem[ZhP16]{ZhPas:2016}
\bysame,
\newblock \textit{Operator estimates in homogenization theory},
\newblock Uspekhi Mat. Nauk, 71 (2016), no.~3, pp.~27--122 (in Russian);
\newblock Russian Math. Surveys, 71 (2016), pp.~417--511 (in English).

\end{thebibliography}

\end{document}